\documentclass[a4paper, leqno]{article}

\usepackage[english]{babel}
\usepackage[utf8x]{inputenc}
\usepackage[T1]{fontenc}

\usepackage[a4paper,top=3cm,bottom=3cm,left=2.8cm,right=2.8cm,marginparwidth=2cm]{geometry}
\usepackage{amsmath}
\usepackage{amsthm}
\usepackage{amsfonts}
\usepackage{braket}
\usepackage{mathtools}
\usepackage{fancyhdr}
\usepackage{graphicx}
\usepackage{tikz}
\usepackage{pgfplots} \pgfplotsset{compat=newest}
\usepackage{subcaption}
\usepackage{hyperref}

\usepackage{enumitem}
\usepackage{textpos}
\usepackage{tipa}
\usepackage{bm}
\usepackage{amssymb}
\usepackage{mathrsfs}
\usepackage[only,llbracket,rrbracket]{stmaryrd}

\usetikzlibrary{quotes,angles}

\newtheorem{theorem}{Theorem}
\newtheorem{corollary}[theorem]{Corollary}
\newtheorem{lemma}[theorem]{Lemma}
\newtheorem{proposition}[theorem]{Proposition}
\newtheorem{recall}[theorem]{Recall}
\newtheorem{definition}[theorem]{Definition}
\newtheorem{assumption}[theorem]{Assumption}
\newtheorem{example}[theorem]{Example}

\theoremstyle{remark}
\newtheorem*{remark}{Remark}

\newcommand{\R}{\mathbf{R}}
\newcommand{\M}{\mathbf{M}}

\newcommand{\N}{\mathbf{N}}

\newcommand{\defeq}{\coloneqq}

\newcommand{\norm}[1]{\left\lVert#1\right\rVert}

\newcommand{\abs}[1]{|#1|}
\newcommand{\tabs}[1]{\Bigl|#1\Bigr|}

\newcommand{\dx}{\mathrm{d}}
\newcommand{\pd}{\partial}

\DeclareMathOperator{\adj}{Adj}
\DeclareMathOperator{\deter}{Det}

\DeclareMathOperator{\supp}{\mathrm{supp}}
\newcommand{\cross}{\times}

\newcommand{\cl}[1]{\overline{#1}}
\newcommand{\ddiv}{{\rm div}}
\newcommand{\grad}{\nabla}
\newcommand{\re}{{\widehat{\rm e}}}

\DeclareMathOperator{\hyp}{\mathrm{hyp}}

\title{On the regularity of continuous solutions to multidimensional scalar conservation laws with \(L^\infty\) source}
\author{Fabio Ancona\thanks{Universit\`a di Padova, 
Emails: \texttt{ancona@math.unipd.it}, \texttt{laura.caravenna@unipd.it}, \texttt{elio.marconi@unipd.it}} 
\and
Laura Caravenna\footnotemark[1]
\and
Alexander J.~Cliffe\thanks{Email: \texttt{ajc314@cantab.ac.uk}}
\and
Elio Marconi\footnotemark[1]}

\date{July 31, 2025\\
\bigskip\bigskip
{\sl Dedicated to Prof.~Shih-Hsien Yu on the occasion of his 60th birthday}
}

\begin{document}

\maketitle
\begin{abstract}
We prove the H\"older regularity of continuous isentropic solutions to multi-dimensional scalar balance laws when the source term is bounded and the flux satisfies general assumptions of nonlinearity.
The results are achieved by exploiting the kinetic formulation of the balance law.
\end{abstract}
\tableofcontents

\section{Introduction}

In this paper we study the regularizing effect of nonlinearity on continuous solutions to balance laws
\begin{align} \label{eq:cl}
    \pd_t u + \ddiv_{\bm{x}}f(u) = g ,
\end{align}
where \(f : \R \to \R^d\) is a smooth flux, \(g \in L^\infty (\R \times \R^d)\) is a bounded source, and $u=u(t,x)$, $x\in \R^d$, is the unknown.

In one space dimension $d=1$, and with zero source $g$, it was established in~\cite{Mar-reg-est} that the nonlinearity of the flux $f$ induces a fractional regularity of distributional solutions to~\eqref{eq:cl} which satisfy entropy inequalities
\begin{equation}\label{E:entr}
    \pd_t \eta(u) + \pd_x q(u)\leq 0\qquad\text{in}\ \ \mathcal{D}_{t,x}',
\end{equation}
associated to any entropy-entropy flux pair \((\eta,q)\), where $\eta(u)$ is convex and $q'(u)= \eta'(u) f'(u)$ for every $u$.

As a technical tool to provide stronger regularity to dispersive PDEs in one space dimension $d=1$, such as the Hunter-Saxton equation and the Camassa-Holm equation, Dafermos began investigating the method of characteristics for \emph{continuous} solutions to a balance law~\eqref{eq:cl}, focusing on the quadratic flux and assuming that the source $g$ was continuous in the single space variable.
When $g$ is zero, thus for a conservation law, he established Euclidean Lipschitz regularity of $f'(u)$ for general smooth fluxes $f$. In particular, Dafermos established that continuous solutions to a conservation law are isentropic: they do not incur entropy production since~\eqref{E:entr} holds as an equality.

Motivated by applications to the analysis of intrinsic Lipschitz graphs in the Heisenberg group, the effect of nonlinearity of the flux on continuous solutions to~\eqref{eq:cl} in one space dimension $d=1$ was studied in~\cite{BCSC} for the quadratic flux and bounded sources $g$. This work followed earlier research on intrinsic regular graphs~\cite{ASCV} relative to continuous sources $g$. Specifically, it was proved in~\cite{BCSC}, by generalizing a lemma on ODEs already present in~\cite{BSCCV2010}, that every continuous solution to the one-dimensional balance law~\eqref{eq:cl} with quadratic flux is $\frac{1}{2}$-Hölder continuous. 
This Hölder continuity result was then generalized to $\frac{1}{\ell}$-Hölder continuity if the flux $f$ is nonlinear of order $\ell$, see~\cite{Car-alphaconvex,Car-Mar-Pin}.

The latter result, alongside a reduction argument allowing for inflection points, contains a finer analysis of the Hölder continuity constant on small balls: it contains a direct proof that this constant is almost everywhere vanishingly small as the radius vanishes, based on the analysis with the Lebesgue differentiation theorem in~\cite{Car-alphaconvex}.
An indirect, much more complex argument was provided by~\cite{BSCCV2010,FSSC2011,CMPSC2014}.

We stress that continuous solutions to~\eqref{eq:cl}, even when $d=1$ and the flux is quadratic, can be very irregular from the Euclidean point of view: continuous solutions to~\eqref{eq:cl} with smooth \emph{convex} fluxes $f$ where the $\alpha$-nonlinearity condition fails are generally not H\"older continuous of any exponent, see~\cite{ABC3}.
Moreover, their graph can be a fractal set and they generally do not have bounded variation, see~\cite{KSC,ABC3}.
Nevertheless, continuous solutions to the one-dimensional balance law~\eqref{eq:cl} with any smooth flux are isentropic, as shown in~\cite{ABC1} where it was proved by approximation with functions of bounded variation.

In the case of more space dimensions, only particular cases are treated, e.g.~\cite{BSCCV2010} analyses a multidimensional system by a reduction argument to the one dimensional setting.
See also~\cite{ADD2023,Vittone}.

%
%
%
%

\vskip.3\baselineskip
This paper exploits the kinetic formulation to 
provide a first 
extension of 
 the H\"older regularity results 
to a real multidimensional setting.
Namely, let \(u : \R \times \R^d \to [0,1]\) be a bounded, {continuous} solution of the scalar balance law~\eqref{eq:cl}, and
assume that \(u\) is isentropic in the sense that
\begin{align}
    \pd_t \eta(u) + \ddiv_{\bm{x}} q(u) = \eta'(u) g \qquad\text{in}\ \ \mathcal{D}_{t,x}',
\end{align}
for any entropy-entropy flux pair \((\eta,q)\).

In the kinetic formulation, for a continuous isentropic solution, there holds (see~\cite{LPT,perthamebook})
\begin{align}\label{eq:kinetic}
    \pd_t \chi + f'(v) \cdot \grad_{\bm{x}}\chi = \mu_{0},
    \qquad \mu_{0}\defeq g \mathcal{L}^{d+1} \otimes \delta_{\{v = u(t,\bm{x})\}},
\end{align}
in the sense of distributions,
where \(\chi : \R \times \R^d \times \R \to \{0,1\}\) is given by
\begin{align}
    \chi(t,x,v) \defeq \begin{cases}
        1 &\quad \text{if \(0 < v \leq u(t,x)\),}\\
        0 &\quad \text{otherwise.}
    \end{cases}
\end{align}

We will make the following standard assumption on the flux \(f \in C^1([0,1];\R^d)\), which quantifies the extent of nonlinearity of the flux.
\begin{assumption} \label{ass:nl1}
	There exists \(\alpha \in (0,1]\) and \(C >0\) such that, for every \((\tau,\xi) \in \R^{d+1}\) with \(\abs{(\tau,\xi)} = 1\), and any \(\delta > 0\), we have
	\begin{align} \label{eq:nonlinearity}
	\begin{aligned}
			\mathcal{L}^1 \big( \{ v \in [0,1] : \abs{1 \cdot \tau + {f'(v)} \cdot \xi} < \delta \} \big) \leq C \delta^\alpha . 
	\end{aligned}
	\end{align}
\end{assumption}

The main result of the paper is:
\begin{theorem}
\label{thm:alpha}
    Let \(u \in C(\R\times\R^d ; [0,1])\) be a continuous and isentropic solution of the scalar balance law~\eqref{eq:cl} with bounded measurable source \(g \in L^\infty(\R\times \R^d)\), and flux \(f \in C^2([0,1];\R^d)\) satisfying the  $\alpha$-nonlinearity Assumption~\ref{ass:nl1}.
    Then, for any \(t\in \R\), \(u(t,\cdot)\) is \(\overline{\gamma}\)-H{\"o}lder continuous in space with \[\overline{\gamma} \defeq \frac{\alpha}{2d}\,.\]
\end{theorem}

We next refine this result making the following more restrictive nonlinearity assumption.

\begin{assumption} \label{ass:nl2}
The flux \(f \in C^{d+2}([0,1];\R^d)\) satisfies that, for every \(v \in [0,1]\), \(\{f''(v), \cdots, f^{(d+1)}(v)\}\) is a spanning set of \(\R^d\), or equivalently that 
\(\{(1,f'(v)), \cdots, (0,f^{(d)}(v)), (0,f^{(d+1)}(v))\}\) is a spanning set of \(\R \times \R^d\).
\end{assumption}
Notice that this is the maximal nonlinearity assumption compatible with the smoothness of the flux, meaning that it implies Assumption \ref{ass:nl1} with $\alpha=\frac1d$ and there are no smooth fluxes $f:\R \to \R^d$ which satisfy \eqref{eq:nonlinearity} with $\alpha> \frac1d$, see Lemma~\ref{L:Silv} below.
Assumption~\ref{ass:nl2} is satisfied by the Burgers flux \[f(u)=\bigg(\frac{u^{2}}2\,,\frac{u^{3}}3\,,\dots\,,\frac{u^{d+1}}{d+1}\bigg) .\]

\begin{theorem}
\label{thm:nondeg}
Let \(u \in C(\R\times\R^d ; [0,1])\) be a continuous and isentropic solution of the scalar balance law~\eqref{eq:cl} with bounded measurable source \(g \in L^\infty(\R\times \R^d)\), and flux \(f \in C^{d+2}([0,1];\R^d)\) satisfying the non degeneracy Assumption~\ref{ass:nl2}.
	Then, for any $t\in\R$, \(u(t,\cdot)\) is  \(\overline{\gamma}\)-H{\"o}lder continuous in space with \[\overline{\gamma} \defeq{\frac{1}{2d}}\,.\]
\end{theorem}

\vskip.3\baselineskip
We thus gain in the exponent $\overline \gamma$ a factor of $\frac 1 \alpha = d$.
The paper is organized as follows:

In Section~\ref{S:Burg1d}, we analyze the one-dimensional case with the Burgers flux and obtain in Corollary~\ref{cor:12holder} the $\frac{1}{2}$-H{\"o}lder regularity for continuous isentropic solutions of~\eqref{eq:cl}. 
Although the results of this section are contained in those following, we feel it is important to begin by proving this case due to the relevance of the Burgers equation in applications, and due to the fact that the proofs are more easily readable.

In Section~\ref{L:mdclNonl}, we address the multidimensional case under the $\alpha$-nonlinearity Assumption~\ref{ass:nl1}. The main result of this section is Proposition~\ref{cor:md-holder}, which proves Theorem~\ref{thm:alpha}.

In Section~\ref{sec:mdcl}, under the nonlinearity Assumption~\ref{ass:nl2}, we improve the H{\"o}lder continuity exponent in Proposition~\ref{cor:md-holder2}, thereby proving Theorem~\ref{thm:nondeg}.
As in \S\ref{S:Burg1d}, we first present the proof for the case of the Burgers equation, before generalizing it.
We also include several examples to illustrate our results.



%
%

\newpage

\begin{table}[h]
\centering
\begin{tabular}{l |r}
\hline
\textbf{Object} & \textbf{Symbol} \\
\hline
\\
Real valued square matrices of size $d\times d$\dotfill & \dotfill$\M_{d\times d}$ \\
Real numbers\dotfill & \dotfill$\R$ \\
Natural numbers, starting from $1$\dotfill & \dotfill$\N$ \\
If $a,b\in\R$, then $a\wedge b$ is the minimum and $a\vee b$ is the maximum\dotfill & \dotfill$\wedge$, $\vee$ \\
$\ell$-th coordinate unit vector in $\R^{d}$, e.g.~$\re_{1}=(1,0,\dots,0)\in\R^{d}$\dotfill &\dotfill$\re_{\ell}$\\
Euclidean norm $\sqrt{a_{1}^{2}+\dots+a_{d}^{2}}$ of a vector $a\in\R^{d}$\dotfill & \dotfill$\norm{a}$ or simply $\abs{a}$\\
Maximum norm $\max\{\abs{a_{1}},\dots,\abs{a_{d}}\}$ of a vector $a\in\R^{d}$\dotfill & \dotfill$\norm{a}_\infty$ or $\abs{a}_\infty$\\
Open ball $\{\norm{y-x}<r\}$ of radius $r$ centered at $x\in\R^{d}$\dotfill &\dotfill \dotfill$B_{r}(x)$ \\
Open square $\{\norm{y-x}_\infty<r\}$ of radius $r$ centered at $x\in\R^{d}$\dotfill &\dotfill 
\(Q_r(x)\)\\
The Lebesgue measure $\omega_d r^d$ of $B_r(0)$, or $(2r)^d$ for the hyper-cube $Q_r(0)$\dotfill &\dotfill 
\(\abs{B_r}\), \(\abs{Q_r}\)\\
The Lebesgue measure of a Lebesgue measurable set $E\subset \R^d$\dotfill &\dotfill 
 \(\abs{E}\)\\
Operator norm of a matrix $A\in\M_{d\times d}$ as in~\eqref{E:opNorm}\dotfill & \dotfill$\norm{A}_{}$\\
Frobenius norm of a matrix $A\in\M_{d\times d}$ as in~\eqref{e:FrobNorm}\dotfill &\dotfill $\norm{A}_{F}$\\
Determinant of a matrix $A\in\M_{d\times d}$, see~\cite[Page~168]{matrici}\dotfill &\dotfill \dotfill$\deter A$\\
Transpose of a matrix $A\in\M_{d\times d}$, see~\cite[Definition~2.17]{matrici}\dotfill &\dotfill \dotfill$  A^{T}$\\
Inverse of a matrix $A\in\M_{d\times d}$, see~\cite[Definitions~2.18;6.9, Proposition~6.7]{matrici}\dotfill &\dotfill \dotfill$  A^{-1}$\\
Topological closure of a set $S\subset\R^{d}$\dotfill & \dotfill$\overline S$ \\
Flux function of a balance law as in~\eqref{eq:cl}, see also Assumptions~\ref{ass:nl1}-\ref{ass:nl2}\dotfill &\dotfill $f$ \\
Order of nonlinearity of the flux function in Assumption~\ref{ass:nl1}\dotfill &\dotfill $\alpha$ \\
Bounded source functions of a balance law as in~\eqref{eq:cl}\dotfill & \dotfill$g$ \\
The uniform norm $\mathrm{ess}\sup \{\norm{p(\xi)}\ :\ \xi\in D\}$ of a function $p:D\subset \R^{m}\to\R^{n}$\dotfill & \dotfill   $\norm{p}_{L^\infty}$\\
Continuous solution of a balance law as in~\eqref{eq:cl} \dotfill&\dotfill $u$ \\
Function~\eqref{def:h} or~\eqref{def:h-cube} measuring oscillations of $u$ at scale $r$ around $x$ at time $t$ \dotfill&\dotfill $h=h_{r}(t,x)$\\
Hypograph of a function $u$, as defined in~\eqref{E:ipo}\dotfill & \dotfill$\hyp(u)$ \\
Space average \(\frac{1}{\abs{E}} \int_E u(t,z) \, \dx z\) of $u$ on any measurable \(E \subset \R \) \dotfill&\dotfill $(u(t,\cdot))_E$\\
Measure of the $v$-super-level set of $u$, within $B_r(y)$, as defined in~\eqref{E:subM}\dotfill & \dotfill$m(y,v)$ \\
Kinetic quantities, see~\eqref{eq:kinetic}, usually denoted as\dotfill &\dotfill $\chi,v,\mu_{0}$ \\
Free kinetic transport as in~\eqref{e:FT}\dotfill & $   \dotfill \mathsf{FT}$ \\
The matrix $\left[
\left(\frac{ih}{d}\right)^{\ell}
\right]_{\ell,i=1,\dots,d}$ of Lemma~\ref{L:normaH}, rows indexed by $\ell$, columns by $i$ \dotfill & \dotfill$H_d(h)$\\
The `almost Wronskian' matrix in~\eqref{e:wroB} or in~\eqref{e:wro} \dotfill & \dotfill$W$ or $W(f,v)$\\
The matrix of increments in~\eqref{e:Aincrements}  \dotfill & \dotfill$A(v,h,d)$\\
\\
\hline
\end{tabular}
\caption{List of Symbols}
\label{tab:notazioni_latex}
\end{table}



\section{One-Dimensional Burgers Case}
\label{S:Burg1d}
We begin with the \(1\)-D case for the Burgers equation, that is \(d=1\) with \(f(u) \defeq \frac12 u^2\).

\begin{definition} \label{def:h}
    For any \(r > 0\), and \((t,x) \in \R\times\R\), we write
    \begin{align}\label{e:h}
        h_r (t,x) \defeq \frac12 \max_{y_1,y_2 \in \overline{B_r(x)}} \big( (u(t,\cdot))_{B_r(y_1)} - (u(t,\cdot))_{B_r(y_2)} \big) , 
    \end{align}
    where \((u(t,\cdot))_E \defeq \frac{1}{\abs{E}} \int_E u(t,z) \, \dx z\) for any measurable \(E \subset \R \).
\end{definition}

\begin{lemma} \label{lem:h-estimate}
    Let \(u \in C( \R \times \R ; [0,1])\) be a continuous and isentropic solution of the Burgers equation with \(L^\infty\)-source \(g\), that is
    \begin{align}
        \pd_t u + \pd_x(\frac12 u^2) = g 
    \end{align}
    in the sense of distributions.
    Then there exists a constant \(C > 0\), depending only on \(\norm{g}_{L^\infty(\R\times\R)}\) such that
    \begin{align} \label{eq:h-estimate}
        h_r (t,x) \leq C r^{\frac13}
    \end{align}
    uniformly in \((t,x) \in \R \times \R \).
    The constant \(C\) can be chosen to satisfy \(C \leq (24 \norm{g}_{L^\infty(\R \times \R)})^{\frac13}\).
    \end{lemma}

\textit{Proof of Lemma~\ref{lem:h-estimate}.}
\textbf{1.}
Fix \((t_0, x_0) \in \R \times \R\).
Let \(y_1, y_2 \in \cl{B_r(x_0)}\) be chosen to realise the maximum in the definition of \(h=h_r(t_0,x_0)\).
We may assume \(h > 0\) otherwise we have nothing to show.
Write
\begin{align}\label{E:subM}
    m (y,v) \defeq \mathcal{L}^1\big( \{ z \in B_r(y) : u(t_0,z) > v \} \big) .
\end{align}
Then 
\begin{align}
    \begin{aligned}
        2 h \abs{B_r} &= \int_0^1 \big( m(y_1,v) - m(y_2,v) \big) \, \dx v \\
        &= \int_0^{1-h} \big( m(y_1,v + h) - m(y_2,v) \big) \, \dx v
        + \int_0^h m(y_1,v) \,\dx v
        - \int_{1-h}^1 m(y_2,v) \,\dx v \\
        &\leq 
        \int_0^{1-h} \big( m(y_1,v + h) - m(y_2,v) \big) \, \dx v + h \abs{B_r} .
    \end{aligned}
\end{align}
Therefore, there exists at least one value \(\overline{v} \in (0,1-h)\) such that 
\begin{align} \label{eq:mean-value-estimate}
    m(y_1,\overline{v}+h) - m(y_2,\overline{v}) > h \abs{B_r} .
\end{align}
Define the rectangles \(R_1 \defeq B_r(y_1) \times (\overline{v}+\frac{h}{2}, \overline{v} + h)\), and \(R_2 \defeq B_r(y_2) \times (\overline{v}, \overline{v} + \frac{h}{2}) \).
By the monotonicity of \(v \mapsto m(y,v)\) it follows that
\begin{align} \label{eq:estimate-1}
\begin{aligned}
    \mathcal{L}^2\big( R_1 \cap \hyp u(t_0, \cdot) \big) 
    - \mathcal{L}^2\big( R_2 \cap \hyp u(t_0, \cdot) \big)
    \geq \frac{1}{2} h^2 \abs{B_r} ,
\end{aligned}
\end{align}
where $\hyp(u)$ denotes the hypograph (or subgraph) of a function $u$, meaning that if $u:D\to [0,+\infty)$, $D\subset\R$, then
\begin{equation}\label{E:ipo}
\hyp(u) \defeq \{(x,v)\in D\times [0,+\infty)\ :\ u(x)\geq v\}.
\end{equation}
\textbf{2.}
For a measurable set \(E \subset \R \times \R\), denote the free-transport by time \(s\) of \(E\) by
\begin{align}\label{e:FT}
    \mathsf{FT}(E;s) \defeq \{(z,v) \in \R \times \R : (z - s f'(v), v) \in E\} .
\end{align}
For some \(T > 0\), we would like to estimate 
\begin{align}
    \tabs{ \mathcal{L}^2 \big( \mathsf{FT}(R_i; T) \cap \hyp u(t_0+T,\cdot)  \big) - \mathcal{L}^2 \big( R_i \cap \hyp u(t_0,\cdot)  \big) }
\end{align}
for \(i = 1\), \(2\).
To this end, let \(r' \in (0,r], v' \in (0, \frac{h}{4}]\) be fixed later, and write
\begin{align}
	R_1' \defeq B_{r+r'}(y_1) \times (\overline{v} + \frac{h}{2} - v', \overline{v} + h + v') .
\end{align}
This new rectangle is just a slight enlargement of $R_1$, increasing the space variable of $r'$ and the kinetic variable of $v'$.
Furthermore, write
\begin{align}
	Q_1 \defeq \bigcup_{t \in [t_0, t_0 + T]} \{t\} \times \mathsf{FT}(R_1'; t-t_0) .
\end{align}
Then testing the kinetic formulation~\eqref{eq:kinetic} against a test function \(\varphi \in C^1_c(Q_1)\) we obtain
\begin{align} \label{eq:integral-estimate1}
\begin{aligned}
	\int_{\mathsf{FT}(R_1';T)} \chi(t_0+T, z, v) \varphi(t_0 + T, z, v) \, \dx z \dx v 
	- \int_{R_1'} \chi(t_0, z, v) \varphi(t_0, z, v) \, \dx z \dx v \\
	= \int_{Q_1} \varphi \, \dx \mu_0 
	+ \int_{Q_1} \chi \big( \pd_t \varphi + f'(v) \pd_x \varphi \big) \, \dx t \dx z \dx v .
\end{aligned}
\end{align}

We should choose \(\varphi \in C_c^1(Q_1)\) carefully so that we can estimate the terms as necessary.
In particular, for \(a, b > 0\), fix a family of cutoff functions \(\phi_{a,b} \in C^1_c([0,\infty);[0,1])\) such that
\begin{align}
	\phi_{a,b}(s) = \begin{cases}
	1 \quad & \text{if \(s \in [0,a]\)} \\
	0 \quad & \text{if \(s \in [a+ \frac{b}{2},\infty)\)}
	\end{cases}
	\qquad {{}-}\frac{4}{b} \leq \phi'(s) \leq 0 .
\end{align}
Then choose
\begin{align}
	\varphi(t, x, v) \defeq \phi_{r,r'}(\abs{x - (t - t_0) f'(v) - y_1}) \, \phi_{\frac{h}{4},v'}(\abs{v-\overline{v} - \frac{3h}{4}}) .
\end{align}
Indeed, this choice satisfies \(\varphi \in C_c^1(Q_1)\), and clearly \(\varphi = 1\) on \(\{t_0\} \times R_1\) and \(\{t_0 + T\}\times \mathsf{FT}(R_1;T)\).
Furthermore, \(\pd_t \varphi + f'(v) \pd_x \varphi = 0\) in \(Q_1\).
Therefore, using also \(0 \leq \chi \leq 1\), the estimate~\eqref{eq:integral-estimate1} becomes
\begin{align*}
\begin{aligned}
	& \tabs{ \mathcal{L}^2 \big( \mathsf{FT}(R_1; T) \cap \hyp u(t_0+T,\cdot)  \big) - \mathcal{L}^2 	\big( R_1 \cap \hyp u(t_0,\cdot)  \big) } \\
	& \qquad \quad
	= \tabs{\int_{\mathsf{FT}(R_1;T)} \chi(t_0+T,z,v) \,\dx z\dx v - \int_{R_1} \chi(t_0,z,v) \, \dx z \dx v } \\
	& \qquad \quad
	\leq \tabs{\int_{Q_1} \varphi \, \dx \mu_0} + 
	\mathcal{L}^2\big( R'_1 \setminus R_1 \big)
	+
	\mathcal{L}^2 \big( \mathsf{FT}(R'_1;T) \setminus \mathsf{FT}(R_1;T) \big) .
\end{aligned}
\end{align*}%
By simple geometric considerations, we have \(\mathcal{L}^2\big( R'_1 \setminus R_1 \big) = h r' + 4 r v' + 4 r' v' \), and \(\mathcal{L}^2 \big( \mathsf{FT}(R'_1;T) \setminus \mathsf{FT}(R_1;T) \big) = \mathcal{L}^2\big( R'_1 \setminus R_1 \big)\).
Using~\eqref{eq:kinetic}, the first term on the right hand side of the inequality above can be estimated by \(\tabs{\int_{Q_1} \varphi \, \dx \mu_0} \leq \norm{g}_{L^\infty(\R \times \R)} \mathcal{L}^2 \big( \pi_{(t,x)} Q_1 \big)\), where \(\pi_{(t,x)} : (t,x,v) \mapsto (t,x)\) is the projection map onto the first two coordinates.
More explicitly, we have
\(\mathcal{L}^2 \big( \pi_{(t,x)} Q_1 \big) = (2(r+r') + \frac12 T (\frac{h}{2} + 2 v')) T \).
{
In particular, sending $r',v'$ to 0, we get
\begin{equation} \label{eq:estimate-2}
\tabs{ \mathcal{L}^2 \big( \mathsf{FT}(R_1; T) \cap \hyp u(t_0+T,\cdot)  \big) - \mathcal{L}^2 	\big( R_1 \cap \hyp u(t_0,\cdot)  \big) } \le \norm{g}_{L^\infty} \Big(2rT + \frac{h}{4}T^2 \Big).
\end{equation}}
The same estimate also holds for \(R_2\).

\textbf{3.}
We choose \(T = \frac{2(y_2-y_1)}{h} \leq \frac{4 r}{h}\), so that \(\mathsf{FT}(R_1;T)\) is ``above'' \(\mathsf{FT}(R_2;T)\), meaning that \(\mathsf{FT}(R_1;T) = (0,\frac{h}{2}) + \mathsf{FT}(R_2;T)\).
Then, for this value of \(T > 0\), using the previous estimates~\eqref{eq:estimate-1}, \eqref{eq:estimate-2}, we have
{
\begin{align}
\begin{aligned}
	0 &\leq \mathcal{L}^2 \big( \mathsf{FT}(R_2; T) \cap \hyp u(t_0+T,\cdot)  \big) - \mathcal{L}^2 	\big( \mathsf{FT}(R_1;T) \cap \hyp u(t_0 + T,\cdot)  \big) \\
	&\leq 2  \left(2rT + \frac{h}{4}T^2 \right)\norm{g}_{L^\infty} - rh^2 \\
	&\leq 24 \norm{g}_{L^\infty} \frac{r^2}{h} - r h^2.
\end{aligned}
\end{align}
Therefore 
\begin{align}
	h^3 \leq 24 \norm{g}_{L^\infty}  r.
\end{align}}
\qed

\begin{lemma} \label{lem:holder}
    Let \(u \in C(\R; \R)\) and $\gamma \in (0,1)$ be such that 
    \begin{align} \label{eq:h-gamma-estimate}
        h_r (x) \leq C r^\gamma \qquad \mbox{for every }x \in \R \mbox{ and }r>0.
    \end{align}
    Then \(u\) is \(\gamma\)-H{\"o}lder with constant at most \(2 C \big( 1 + \frac{2}{2^\gamma -1} \big)\).
    If \(\gamma \in [\frac13,\frac12)\), the H{\"o}lder constant satisfies \(2 C \big( 1 + \frac{2}{2^\gamma -1} \big) \leq 10 C\).
\end{lemma}

\textit{Proof of Lemma~\ref{lem:holder}.}
For any \(x , y \in \R\), let \(r \defeq \abs{y - x}\).
We have
\begin{align}
\begin{aligned}
    \abs{u(x) - u(y)} \leq \abs{u(x) - (u)_{B_r(x)}} + \abs{(u)_{B_r(x)} - (u)_{B_r(y)}} + \abs{u(y) - (u)_{B_r(y)}} \\
    \leq 
    \abs{u(x) - (u)_{B_r(x)}} + \abs{u(y) - (u)_{B_r(y)}} + 2 h_r(x) .
\end{aligned}
\end{align}
Next we have the simple estimate
\begin{align}
\begin{aligned}
    \abs{(u)_{B_r(x)} - (u)_{B_{\frac{r}{2}}(x)}}
    &=
    \tabs{\frac{1}{2r} \int_{x-r}^{x+r} u(z) \,\dx z - \frac{1}{r} \int_{x-\frac{r}{2}}^{x+\frac{r}{2}} u(z) \,\dx z} \\
    &\leq \frac{1}{2r} \tabs{ \int_{x-r}^{x} u(z) \,\dx z - \int_{x-\frac{r}{2}}^{x+\frac{r}{2}} u(z) \,\dx z }
    + \frac{1}{2r} \tabs{ \int_{x}^{x+r} u(z) \,\dx z - \int_{x-\frac{r}{2}}^{x+\frac{r}{2}} u(z) \,\dx z } \\
    &\leq 2 h_{\frac{r}{2}}(x) .
\end{aligned}
\end{align}
Since \(u\) is continuous, we have
\begin{align}
\begin{aligned}
    \abs{(u)_{B_r(x)} - u(x)} = \tabs{\sum_{k=0}^\infty \Big( (u)_{B_{\frac{r}{2^k}}(x)} - (u)_{B_{\frac{r}{2^{k+1}}}(x)} \Big) }
    \leq \sum_{k = 0}^\infty 2 h_{\frac{r}{2^{k+1}}}(x) 
    \leq 2 C \sum_{k=0}^\infty \big( \frac{r}{2^{k+1}} \big)^\gamma
    \\
    = \frac{2 C r^\gamma}{2^\gamma - 1 } .
\end{aligned}
\end{align}
Therefore 
\begin{align}
    \abs{u(x) - u(y)}  \leq 2 h_r (x) + \frac{4C r^\gamma}{2^\gamma -1}
    \leq 2 C \big( 1 + \frac{2}{2^\gamma -1} \big) \abs{y-x}^\gamma .
\end{align}
\qed

%


\begin{lemma} \label{lem:improved-holder}
    Let \(u \in C(\R \times \R; [0,1])\) be as in Lemma~\ref{lem:h-estimate}.
	For some \(t \in \R\) and \(\gamma \in (0,\frac12)\), assume that
	\begin{align}
		h_r(t,x) \leq C r^\gamma
	\end{align}
	uniformly in \(x \in \R\).
	Then in fact, there exists a constant \(\widetilde{C} >0\), depending only on \((C,\gamma,\norm{g}_{L^\infty(\R\times\R)})\), such that  
	\begin{align}
		h_r(t,x) \leq \widetilde{C} r^{\frac{1+\gamma}{3}}
	\end{align}
	uniformly in \(x\in\R\).
	If \(\gamma \in [\frac13,\frac12)\), the constant \(\widetilde{C}\) can be chosen to satisfy \(\widetilde{C} \leq 10 (2 C \norm{g}_{L^\infty(\R \times \R)})^{\frac13}\).
\end{lemma}

\textit{Proof of Lemma~\ref{lem:improved-holder}.}
\textbf{1.}
Let \(t \in \R\) be as given, and fix any \(x \in \R\).
Set 
\begin{align}
	\underline{u} \defeq {
    \min_{y \in \overline{B_{2r}(x)}} u(t,y)} , \qquad 
		\overline{u} \defeq{
        \max_{y \in \overline{B_{2r}}(x)} u(t,y)} .
\end{align}
Once again, let \(y_1, y_2 \in \cl{B_r(x)}\) be chosen to realise the maximum in the definition of \(h=h_r(t,x)\).
Then
\begin{align}
\begin{aligned}
	2 h \abs{B_r} &= \int_{\underline{u}}^{\overline{u}} \big( m(y_1,v) - m(y_2,v) \big) \,\dx v \\
	&= \int_{\underline{u}}^{\overline{u}-h} \big( m(y_1,v+h) - m(y_2,v) \big) \, \dx v 
	+ \int_{\underline{u}}^{\underline{u}+h} m(y_1,v) \, \dx v
	- \int_{\overline{u} - h}^{\overline{u}} m(y_2,v) \, \dx v \\
	&\leq \int_{\underline{u}}^{\overline{u}-h} \big( m(y_1,v+h) - m(y_2,v) \big) \, \dx v
	+ h \abs{B_r} ,
\end{aligned}
\end{align}
so there exists at least one value \(\overline{v} \in (\underline{u},\overline{u}-h)\) such that 
\begin{align}
	m(y_1,\overline{v} + h) - m(y_2, \overline{v}) \geq \frac{h}{\overline{u}-\underline{u}-h} \abs{B_r} .
\end{align}
Note that trivially \(0 \leq 2h \leq \overline{u} - \underline{u}\), and from Lemma~\ref{lem:holder}, \(\overline{u} - \underline{u} \leq 2 C( 1 + \frac{2}{2^\gamma-1}) (4r)^\gamma\). 
Therefore
\begin{align} \label{eq:improved-mean-value-estimate}
	m(y_1,\overline{v} + h) - m(y_2, \overline{v}) \geq \frac{2^\gamma -1}{2 \cdot 4^\gamma (2^\gamma +1)} \frac{h}{C r^\gamma} \abs{B_r} .
\end{align}

\textbf{2.}
We are now free to repeat the same argument as Steps~\textbf{2}--\textbf{3} from the proof of Lemma~\ref{lem:h-estimate} using~\eqref{eq:improved-mean-value-estimate} in place of~\eqref{eq:mean-value-estimate}.
Doing so we obtain
{
\begin{align}
0 &\leq 24 \norm{g}_{L^\infty}\frac{r^2}h - \frac{2^\gamma -1}{2 C \cdot 4^\gamma (2^\gamma + 1)} r^{1-\gamma} h^2 
\end{align}
and therefore 
\begin{align}
	h^3 \leq 24  \frac{2 C \cdot 4^\gamma (2^\gamma +1)}{2^\gamma-1} \norm{g}_{L^\infty(\R \times \R)}  r^{1+\gamma}.
\end{align}
%
}
\qed

\vskip\baselineskip
The following statement is the main result of this section.

\begin{corollary} \label{cor:12holder}
Let \(u \in C(\R \times \R; [0,1])\) be as in Lemma~\ref{lem:h-estimate}.
Then \(u(t,\cdot)\) is \(\frac12\)-H{\"o}lder with constant at most \(100 \cdot \big( 20 \, \norm{g}_{L^\infty(\R\times\R)} \big)^{\frac12}\).
\end{corollary}

\textit{Proof of Corollary~\ref{cor:12holder}.}
Fix \((t,x) \in \R \times \R\).
By Lemma~\ref{lem:h-estimate}, we have that \(h_r(t,x) \leq C_1 r^{\frac13}\), with \(C_1 \leq (24\norm{g}_{L^\infty(\R\times\R)})^{\frac13}\).
Starting from this estimate, we can apply Lemma~\ref{lem:improved-holder} iteratively to obtain \(h_r(t,x) \leq C_n r^{\gamma_n}\) for \(n \in \N\), where \(\gamma_n \defeq \frac12 (1- 3^{-n})\), and the constants \(C_n\) satisfy
\begin{align}
\begin{aligned}
	C_{n+1} \leq 10 \cdot (2 C_n \norm{g}_{L^\infty(\R\times\R)})^{\frac13} 
	&\leq C_1^{3^{-n}} \cdot \big( 2000 \, \norm{g}_{L^\infty(\R\times\R)} \big)^{\frac12(1-3^{-n})} \\
	&\leq 10 \sqrt{20} \cdot ( 1 \vee C_1 ) \cdot \big( 1 \vee \norm{g}_{L^\infty(\R \times \R)}^{\frac12}  \big) .
\end{aligned}
\end{align}
Using Lemma~\ref{lem:holder}, for each \(n\in \N\), we have that \(u(t,\cdot)\) is \(\gamma_n\)-H{\"o}lder with constant at most \(10 C_n\), which is uniformly bounded by the above.
Therefore we can simply pass to the limit to obtain that \(u(t,\cdot)\) is \(\frac12\)-H{\"o}lder, with constant at most \(10 \cdot \limsup\limits_{n\to\infty} C_n \leq 10 \cdot \big( 2000 \, \norm{g}_{L^\infty(\R\times\R)} \big)^{\frac12}\).%
\qed

\section{$\alpha$-nonlinear Multi-Dimensional Case} \label{L:mdclNonl}
We begin by modifying slightly a general decomposition lemma, which can be found in~\cite[Lemma~4.2]{Mar-str-wsol}, to a form that is directly applicable to our problem.
\begin{lemma}[{\cite[Lemma~4.2]{Mar-str-wsol}}] \label{lem:decomposition-lemma}
	Let \(f \in C^1([0,1];\R^d)\) satisfy the standard nonlinearity Assumption~\ref{ass:nl1}.
	There exists a constant \(\widetilde{C}\), depending only on \((d,\norm{f'}_{L^\infty},C)\) such that the following statement holds.
	For any \(v , h  > 0\), with \(v + h < 1\), there exist \(v_1 < \cdots < v_{d+1} \in [v, v+h]\) satisfying \(\abs{v_{i+1} - v_i} \geq \frac{h}{2(d+1)}\) such that, for every \(a \in \R^d\), there exist \(a_1, \cdots, a_{d+1} \in \R\) with
	\begin{align} \label{eq:estimate-of-coefficients}
		a = \sum_{i=1}^{d+1} a_i f'(v_i),
		\qquad 
		0 = \sum_{i=1}^{d+1} a_i,
		\qquad
		\abs{a_i} \leq \widetilde{C} \frac{\abs{a}}{h^{d/\alpha}} .
	\end{align}	
\end{lemma}

\textit{Proof of Lemma~\ref{lem:decomposition-lemma}.}
The proof proceeds exactly as~\cite[Lemma~4.2]{Mar-str-wsol} applied to the flux \(v \mapsto (v,f(v))\) in \(\R \times \R^d\).
Indeed, we obtain appropriate values \(v_1 < \cdots < v_{d+1} \in [v,v+h]\), satisfying \(\abs{v_{i+1} - v_i} \geq \frac{h}{2(d+1)}\), such that \(\{(1,f'(v_i)) : 1 \leq i \leq d+1\}\) is a spanning set of \(\R \times \R^d\).
Then, for any \((0,a) \in \R \times \R^d\), there exist coefficients \(\{a_i \in \R : 1 \leq i \leq d+1\}\) with \((0,a) = \sum_{i=1}^{d+1} a_i (1,f'(v_i))\).
The slightly improved estimate on the coefficients \(\{a_i\}\) stems from the fact that \(\sum_{i=1}^{d+1} a_i = 0\), and therefore \(\abs{a_1} = \abs{\sum_{i = 2}^{d+1} a_i}\), meaning that there is no increase in magnitude during the final step of the iterated estimates.
\hfill \qed

Next, we modify the estimate~\cite[Lemma~4.3]{Mar-str-wsol} to a form that is directly applicable to our problem, taking into account also the small change in Lemma~\ref{lem:decomposition-lemma} above.
For later purposes, it is also necessary to modify slightly the definition of \(h_r\) given in~\cite[Eq.~(4.4)]{Mar-str-wsol} as follows.
\begin{definition} \label{def:h-cube}
	For any \(r > 0\), and \((t,x) \in \R \times \R^d\), we write
    \begin{align}
        h_r (t,x) \defeq \frac12\, {
        \max_{y_1,y_2 \in {\overline{Q_r(x)}}} \Big( (u(t,\cdot))_{Q_r(y_1)} - (u(t,\cdot))_{Q_r(y_2)} \Big) }\in \Big[0,\frac12\Big], 
    \end{align}
    where \((u(t,\cdot))_E \defeq \frac{1}{\abs{E}} \int_E u(t,y) \, \dx y\) for any measurable \(E \subset \R^d \), and \(Q_r(x) \defeq \{ y \in \R^d : \abs{y-x}_\infty < r \}\).
\end{definition}

\begin{lemma}[{\cite[Lemma~4.3]{Mar-str-wsol}}] \label{lem:estimate-cylinder}
	There exists a constant \(C_\omega > 0\), depending only on \((d, \norm{f''}_{L^\infty}, \widetilde{C}) \) such that, for any \(r > 0\), and \((t,x) \in \R \times \R^d\), there exists \((\tilde{t},\tilde{x}) \in \R \times \R^d\), \(\tilde{v} \in [0,1)\), and \(\tilde{a} \in \R\) satisfying
	\begin{align} \label{eq:estimate-of-tildes}
		\abs{\tilde{a}} \leq \widetilde{C} \frac{2 r \sqrt{d}}{h_r(t,x)^{d/\alpha}},
		\qquad
		\abs{(\tilde{t},\tilde{x}) - (t,x)} \leq r \sqrt{d} + (d+1) (1 + \norm{f'}_{L^\infty}) \widetilde{C} \frac{2 r \sqrt{d}}{h_r(t,x)^{d/\alpha}},
	\end{align}
	such that, for any \(\omega \in (0, C_\omega \, h_r(t,x)^{1+\frac{d}{\alpha}}]\), we have
	\begin{align} \label{eq:estimate-cylinder}
		\mathcal{L}^{d+1}\big(\mathcal{Q} \cap \hyp u(\tilde{t}) \big) - 
		\mathcal{L}^{d+1}\big(\mathcal{Q} \cap \mathsf{FT}(\hyp u(\tilde{t}-\tilde{a}) ; \tilde{a})\big)
		\geq \frac{\omega h_r(t,x) \abs{Q_r}}{2(d+1)} ,
	\end{align}
	where the cuboid \(\mathcal{Q} \defeq Q_r(\tilde{x}) \times [\tilde{v},\tilde{v} + \omega]\).
\end{lemma}

\textit{Proof of Lemma~\ref{lem:estimate-cylinder}.}
\textbf{1.}
Fix \(r>0\), \((t,x) \in \R \times \R^d\), and let \(y_1, y_2 \in \cl{Q_r(x)}\) be chosen to realise the maximum in the definition of \(h=h_r(t,x)\).
In the following, we exclusively write \(h \defeq h_r(t,x)\).

For any \((s,y,v) \in \R \times \R^d \times \R\), write
\begin{align}
	m(s,y,v) \defeq \mathcal{L}^d \big( \{ z \in Q_r(y) : u(s,z) > v \} \big) .
\end{align}
Then
\begin{align}\label{e:1}
    \begin{aligned}
        2 h \abs{Q_r}
		&= \int_0^1 \big( m(t,y_1,v) - m(t,y_2,v) \big) \, \dx v \\
        &= \int_0^{1-h} \big( m(t,y_1,v + h) - m(t,y_2,v) \big) \, \dx v
        + \int_0^h m(t,y_1,v) \,\dx v
        - \int_{1-h}^1 m(t,y_2,v) \,\dx v \\
        &\leq 
        \int_0^{1-h} \big( m(t,y_1,v + h) - m(t,y_2,v) \big) \, \dx v + h \abs{Q_r} .
    \end{aligned}
\end{align}
Therefore, there exists at least one value \(\overline{v} \in (0,1-h)\) such that 
\begin{align} \label{eq:mean-value-estimate2}
    m(t,y_1,\overline{v}+h) - m(t,y_2,\overline{v}) > h \abs{Q_r} .
\end{align}

\textbf{2.}
Write \(a \defeq y_1 - y_2\).
Using Lemma~\ref{lem:decomposition-lemma}, we find that there exist \(v_1 < \cdots < v_{d+1} \in [\overline{v},\overline{v}+h]\) and coefficients \(\{a_i \in \R : 1 \leq i \leq d+1\}\) such that \(\sum_{i = 1}^{d+1} a_i = 0\), and \( a = \sum_{i=1}^{d+1} a_i f'(v_i)\).
Set \((t_0,x_0) \defeq (t,y_2)\), and iteratively \((t_i,x_i) \defeq (t_{i-1},x_{i-1}) + a_i (1,f'(v_i))\) for \(1 \leq i \leq d+1\).
Then \((t_{d+1},x_{d+1}) = (t,y_1)\) by construction.
Moreover, upon writing \(v_0 \defeq \overline{v}\) and \(v_{d+2} \defeq \overline{v} + h\), we have
\begin{align}
\begin{aligned}
	h\abs{Q_r} < m(t_{d+1},x_{d+1},v_{d+2}) - m(t_0,x_0,v_0) 
	&\leq m(t_{d+1},x_{d+1},v_{d+2}) - m(t_0,x_0,v_1) \\
	&= \sum_{i = 1}^{d+1} \big( m(t_i,x_i,v_{i+1}) - m(t_{i-1},x_{i-1},v_i) \big) .
\end{aligned}
\end{align}
In particular, there exists a choice \(\ell \in \{1, \cdots, d+1 \}\) such that
\begin{align} \label{eq:choice-of-ell}
m(t_\ell,x_\ell,v_{\ell+1}) - m(t_{\ell-1},x_{\ell-1},v_\ell) > \frac{h \abs{Q_r}}{d+1} .
\end{align}
Consequently we set \((\tilde{t},\tilde{x}) \defeq (t_\ell,x_\ell)\), \(\tilde{v} \defeq v_\ell\), and \(\tilde{a} \defeq a_\ell\).
The first estimate in~\eqref{eq:estimate-of-tildes} follows directly from~\eqref{eq:estimate-of-coefficients}, whilst for the second estimate we have
\begin{align}
	(\tilde{t},\tilde{x}) = (t,y_2) + \sum_{i=1}^{\ell} a_i (1, f'(v_i)) ,
\end{align}
and therefore 
\begin{align}
\abs{(\tilde{t},\tilde{x}) - (t,x)} \leq \abs{y_2 - x} + \sum_{i = 1}^{\ell} \abs{a_i} (1 + \abs{f'(v_i)}) \leq r \sqrt{d} + (d+1) (1 + \norm{f'}_{L^\infty}) \widetilde{C} \frac{2 r \sqrt{d}}{h^{d/\alpha}}.
\end{align}

\textbf{3.}
For some \(\omega \in (0, \frac{h}{2(d+1)})\) to be determined later, consider the cuboid \(\mathcal{Q} \defeq Q_r(\tilde{x}) \times [\tilde{v},\tilde{v} + \omega]\).
We can estimate
\begin{align} \label{eq:estimate-cylinder-proof}
\begin{aligned}
	& \mathcal{L}^{d+1}\big( \mathcal{Q} \cap \hyp u(\tilde{t}) \big) - \mathcal{L}^{d+1} \big( \mathcal{Q} \cap \mathsf{FT}( \hyp u(\tilde{t}-\tilde{a}); \tilde{a} ) \big) \\
	&\qquad \qquad \qquad
	= \int_0^{\omega} \big( m(\tilde{t},\tilde{x}, \tilde{v} + v) - m(\tilde{t}-\tilde{a}, \tilde{x} - \tilde{a} f'(\tilde{v}+v), \tilde{v} + v) \big) \, \dx v .
\end{aligned}
\end{align}
Therefore, it is sufficient to estimate the quantity
\begin{align}
	m(\tilde{t},\tilde{x}, v) - m(\tilde{t}-\tilde{a}, \tilde{x} - \tilde{a} f'(v), v)
\end{align}
for any \(v \in [\tilde{v},\tilde{v}+\omega]\).
Using~\eqref{eq:choice-of-ell}, and the monotonicity of \(m\) with respect to \(v\), we have
\begin{align} \label{eq:m-prime-estimate}
\begin{aligned}
	m(\tilde{t},\tilde{x}, v) &- m(\tilde{t}-\tilde{a}, \tilde{x} - \tilde{a} f'(v), v) \\
	&\qquad = m(\tilde{t},\tilde{x}, v) - m(\tilde{t},\tilde{x}, v_{\ell+1}) + m(\tilde{t},\tilde{x}, v_{\ell+1})
	- m(\tilde{t}-\tilde{a},\tilde{x}-\tilde{a}f'(\tilde{v}), \tilde{v}) \\
	&\qquad \quad + m(\tilde{t}-\tilde{a},\tilde{x}-\tilde{a}f'(\tilde{v}), \tilde{v}) 
	- m(\tilde{t}-\tilde{a},\tilde{x}-\tilde{a}f'(v), \tilde{v}) \\
	&\qquad \quad + m(\tilde{t}-\tilde{a},\tilde{x}-\tilde{a}f'(v), \tilde{v}) 
	- m(\tilde{t}-\tilde{a},\tilde{x}-\tilde{a}f'(v), v) \\
	&\qquad \geq \frac{h\abs{Q_r}}{d+1} + m(\tilde{t}-\tilde{a},\tilde{x}-\tilde{a}f'(\tilde{v}), \tilde{v}) -  m(\tilde{t}-\tilde{a},\tilde{x}-\tilde{a}f'(v), \tilde{v}) ,
\end{aligned}
\end{align}
since \(v_\ell = \tilde{v} \leq v \leq \tilde{v}+\omega < v_{\ell+1}\).
Clearly we have
\begin{align}
\begin{aligned}
	\abs{ m(\tilde{t}-\tilde{a},\tilde{x}-\tilde{a}f'(\tilde{v}), \tilde{v}) -  m(\tilde{t}-\tilde{a},\tilde{x}-\tilde{a}f'(v), \tilde{v}) } 
	&\leq \abs{ Q_r(\tilde{x}-\tilde{a} f'(\tilde{v})) \,\triangle \, Q_r(\tilde{x}-\tilde{a} f'(v)) }\\
	&\leq 2( (2r)^d - (2r - \abs{\tilde{a}} \abs{v - \tilde{v}} \norm{f''}_{L^\infty})^d ) \\
	&\leq C_d \, r^{d-1} \abs{\tilde{a}} \abs{v - \tilde{v}} \norm{f''}_{L^\infty},
\end{aligned}
\end{align}
whenever \( \abs{\tilde{a}} \abs{v - \tilde{v}} \norm{f''}_{L^\infty} \leq 2r\), for some absolute constant \(C_d >0\) depending only on \(d\).
Combining this with the first estimate in~\eqref{eq:estimate-of-tildes}, we have
\begin{align}\label{e:altro}
	\abs{ m(\tilde{t}-\tilde{a},\tilde{x}-\tilde{a}f'(\tilde{v}), \tilde{v}) -  m(\tilde{t}-\tilde{a},\tilde{x}-\tilde{a}f'(v), \tilde{v}) } 
	\leq
	C_d \, r^d \, \omega \, \widetilde{C} \frac{2\sqrt{d}}{h^{d/\alpha}} \norm{f''}_{L^\infty}.
\end{align}
Therefore, the estimate~\eqref{eq:estimate-cylinder} follows immediately from~\eqref{eq:estimate-cylinder-proof} and~\eqref{eq:m-prime-estimate} whenever \(\omega > 0\) is chosen small enough so that
\begin{align}\label{e:constraintwomega}
	\omega \leq \frac{h^{d/\alpha}}{\widetilde{C}\sqrt{d}\norm{f''}_{L^\infty}} 
	\wedge 
	\frac{2^d \, h^{1 + d/\alpha} }{4 C_d \widetilde{C} (d+1) \sqrt{d} \norm{f''}_{L^\infty}}
	\wedge \frac{h}{2(d+1)} .
\end{align}
Since \(h \in [0,1]\), it is sufficient to take the constant \(C_\omega\) in~\eqref{eq:estimate-cylinder} to be
\begin{align}
	C_\omega = \frac{2^d}{4 C_d \widetilde{C} (d+1) \sqrt{d} \norm{f''}_{L^\infty}} 
	\wedge \frac{1}{\widetilde{C} \sqrt{d} \norm{f''}_{L^\infty}} 
	\wedge \frac{1}{2(d+1)}.
\end{align}
\qed

In the next lemma, we adapt and improve slightly the estimate~\cite[Proposition~5.6]{Mar-str-wsol} to a form that is directly applicable to our problem.
First, let us fix a family of smooth functions \(\psi_{a,b} \in C^\infty_c([0,\infty); [0,1])\) satisfying 
\begin{align}
	\psi_{a,b}(s) = \begin{cases}
	1 \quad & \text{if \(s \in [0,a]\)} \\
	0 \quad & \text{if \(s \in [a + \frac{b}{2}, \infty)\)}
	\end{cases}
	\qquad -\frac{4}{b} \leq \psi'(s) \leq 0 .
\end{align}


\begin{lemma}[{\cite[Proposition~5.6]{Mar-str-wsol}}] \label{lem:test-function-estimate}
	Let \(u \in C(\R\times\R^d ; [0,1])\) be a continuous and isentropic solution of the scalar balance law~\eqref{eq:cl} with bounded measurable source \(g \in L^\infty(\R\times \R^d)\), and flux \(f \in C^2([0,1];\R^d)\) satisfying the standard non-degeneracy Assumption~\ref{ass:nl1}.
	Fix \(r > 0\), \((\tilde{t},\tilde{x}) \in \R \times \R^d\), \(\tilde{v} \in [0,1)\), \(\omega \in (0,1-\tilde{v}]\), and \(\tilde{a} \in \R\).
	Moreover, for any \(r' \leq r\), \(\omega' \leq \omega\), fix the test function \(\phi \in C^\infty_c(\R^d \times (0,\infty) ; [0,1])\)
	\begin{align} \label{eq:test-phi}
		\phi(x,v) \defeq \prod_{i = 1}^{d} \psi_{r,r'}\big( \abs{(x-\tilde{x}) \cdot \re_i} \big) \psi_{\frac{\omega}{2},\omega'} \big( \abs{v - \tilde{v} - \frac{\omega}{2}} \big) ,
	\end{align}
	with \(\re_i \in \R^d\) the standard unit basis vectors.
	Then
	\begin{align} \label{eq:test-function-estimate}
	\begin{aligned}
		\int_{\R^d \times [0,\infty)} \phi(x,v) \big( \chi_{\hyp u(\tilde{t})} - \chi_{\mathsf{FT}(\hyp u(\tilde{t}-\tilde{a});\tilde{a})} \big) \,\dx x \dx v 
		\leq 
		\abs{\tilde{a}} \norm{g}_{L^\infty} \abs{Q_{r + r' + \abs{\tilde{a}} (\frac{\omega}{2} + \omega') \norm{f''}_{L^\infty} }} .
	\end{aligned}
	\end{align}
\end{lemma}

\textit{Proof of Lemma~\ref{lem:test-function-estimate}.}
The proof follows closely the proof of~\cite[Proposition~5.6]{Mar-str-wsol}, though we specialise the final estimate in a way that is specific to our problem, resulting in a better overall estimate.
We outline the main details below.

\textbf{1.}
Write 
\begin{align}
	\chi_0(t,x,v) \defeq \chi_{\mathsf{FT}(\hyp u(\tilde{t} - \tilde{a}); t - \tilde{t} + \tilde{a})}(x,v),
	\qquad
	\chi_1(t,x,v) \defeq \chi_{\hyp u(t)}(x,v) .
\end{align}
Then a straightforward modification of~\cite[Lemma~3.1]{Mar-str-wsol} tells us that
\begin{align}
	\pd_t \chi_0 + f'(v) \cdot \grad_x \chi_0 = 0
	\qquad \text{in} \ \ \mathcal{D}'(\R\times\R^d\times\R),
\end{align}
whilst from the kinetic formulation~\eqref{eq:kinetic}
we have
\begin{align}
	\pd_t \chi_1 + f'(v) \cdot \grad_x \chi_1 = \mu_0
	\qquad \text{in} \ \ \mathcal{D}'(\R\times\R^d\times\R).
\end{align}
Set \(\tilde{\chi} \defeq \chi_1 - \chi_0\), and \(\tilde{\phi}(t,x,v) \defeq \phi (x - (t-\tilde{t}) f'(v), v)\), then direct computation shows that
\begin{align} \label{eq:chi-tilde-equation}
	\pd_t (\tilde{\phi}\tilde{\chi}) + f'(v) \cdot \grad_x (\tilde{\phi}\tilde{\chi}) = \tilde{\phi} \mu_0
	\qquad \text{in} \ \ \mathcal{D}'(\R\times\R^d\times\R).
\end{align}
For any \(t \in \R\), set
\begin{align}
	p(t) \defeq \int_{\R^d \times [0,\infty)} \tilde{\phi}(t,x,v) \tilde{\chi}(t,x,v) \, \dx x \dx v ,
\end{align}
so that
\begin{align}
	p(\tilde{t} - \tilde{a}) = 0 ,
	\qquad 
	p(\tilde{t}) = \int_{\R^d \times [0,\infty)} \phi(x,v) \big( \chi_{\hyp u(\tilde{t})} - \chi_{\mathsf{FT}(\hyp u(\tilde{t}-\tilde{a});\tilde{a})} \big) \,\dx x \dx v .
\end{align}

Using~\eqref{eq:chi-tilde-equation}, for any \(\varphi \in C^\infty_c( \R )\) it can be checked directly that
\begin{align}
	\langle p', \varphi \rangle = \int_{\R^d \times [0,\infty)} \tilde{\phi}(t,x,v) \varphi(t) \, \dx \mu_0 ,
\end{align}
where \(p'\) is the distributional derivative of \(p\) in \(\mathcal{D}'(\R)\).
Being also that \(t \mapsto p(t)\) is continuous, from the fundamental theorem of calculus for distributions we have
\begin{align}
	p(\tilde{t}) - p(\tilde{t}-\tilde{a}) = \int_{I \times \R^d \times [0,\infty)} \tilde{\phi}(t,x,v) \, \dx \mu_0 ,
\end{align}
where \(I \subset \R\) is the time interval with endpoints \(\tilde{t} - \tilde{a}\) and \(\tilde{t}\).

\textbf{2.}
Now we specialise our estimate to the problem at hand.
Using the form of \(\mu_0\) and \(\phi\) given in~\eqref{eq:kinetic} and~\eqref{eq:test-phi}, we have
\begin{align} \label{eq:test-function-indicator-estimate}
\begin{aligned}
\int_{I \times \R^d \times [0,\infty)} \tilde{\phi}(t,x,v) \, \dx \mu_0 
=
\int_{I \times \R^d} g(t,x) \phi(x - (t-\tilde{t}) f'(u(t,x)), u(t,x)) \, \dx t \dx x \\
\leq 
\norm{g}_{L^\infty} \int_{I \times \R^d } \bm{1}_{x - (t-\tilde{t}) f'(u(t,x)) \in Q_{r+r'}(\tilde{x})} \, \bm{1}_{ \tilde{v} - \omega' \leq u(t,x) \leq \tilde{v} + \omega + \omega'} \, \dx t \dx x .
\end{aligned}
\end{align}
At a given time \(t \in I\), using the interval bound on \(u(t,x)\) given by the second indicator function, in order for the product of the two indicator functions to be non-zero, it is necessary that 
\begin{align}
\begin{aligned}
	r + r' &\geq \abs{x - \tilde{x} - (t-\tilde{t}) f'(u(t,x))}_\infty \\
	&\geq \abs{x - \tilde{x} - (t-\tilde{t}) f'(\tilde{v}+\frac{\omega}{2})}_\infty - 
	\abs{t - \tilde{t}} \abs{f'(u(t,x)) - f'(\tilde{v}+\frac{\omega}{2})}_\infty \\
	&\geq 
	\abs{x - \tilde{x} - (t-\tilde{t}) f'(\tilde{v}+\frac{\omega}{2})}_\infty - 
	\abs{\tilde{a}} (\frac{\omega}{2} + \omega') \norm{f''}_{L^\infty} ,
\end{aligned}
\end{align}
that is to say that \(x \in Q_{r + r' + \abs{\tilde{a}} (\frac{\omega}{2} + \omega') \norm{f''}_{L^\infty}}(\tilde{x} + (t - \tilde{t})f'(\tilde{v} + \frac{\omega}{2}))\).
The given estimate in the lemma follows since \(\abs{I} = \abs{\tilde{a}}\).
\qed



By choosing the constants \(r' \leq r\) and \(\omega' \leq \omega\) sufficiently small in Lemma~\ref{lem:test-function-estimate} so that the test function \(\phi\) approximates closely the characteristic function of the cuboid \(\mathcal{Q}\), we can combine the estimates in Lemmas~\ref{lem:estimate-cylinder}--\ref{lem:test-function-estimate} to deduce the algebraic decay rate of \(r \mapsto h_r(t,x)\), uniformly in \((t,x) \in  \R \times \R^d\).
\begin{lemma} \label{lem:md-h-estimate}
	Let \(u \in C(\R\times\R^d ; [0,1])\) be a continuous and isentropic solution of the scalar balance law~\eqref{eq:cl} with bounded measurable source \(g \in L^\infty(\R\times \R^d)\), and flux \(f \in C^2([0,1];\R^d)\) satisfying the standard nonlinearity Assumption~\ref{ass:nl1}.
	Then, there exists a constant \(C_0\), depending only on \((d,\norm{f''}_{L^\infty},\norm{g}_{L^\infty},\widetilde{C},C_\omega)\), such that 
	\begin{align}\label{e:accessorio}
		h_r(t,x) \leq C_0 r^{\gamma_0},
		\qquad
		\gamma_0 \defeq \frac{1}{2(1+ \frac{d}{\alpha})}
	\end{align}
	for every \((t,x) \in \R \times \R^d\).
\end{lemma}

\textit{Proof of Lemma~\ref{lem:md-h-estimate}.} 
The proof follows closely~\cite[Proposition~4.4]{Mar-str-wsol}, though we present the main steps below.

\textbf{1.}
Fix \(r > 0\), \((t,x) \in \R\times\R^d\), \(h \defeq h_r(t,x)\), \(\omega = C_\omega h^{1+\frac{d}{\alpha}}\), and let \((\tilde{t},\tilde{x}) \in \R \times \R^d\), \(\tilde{v} \in [0,1)\), \(\tilde{a} \in \R\), and \(\mathcal{Q}\) be given as in Lemma~\ref{lem:estimate-cylinder}.
For \(r' \leq r\), and \(\omega' \leq \omega\) to be determined later, let the test function \(\phi\) be given as in Lemma~\ref{lem:test-function-estimate}.

Using Lemma~\ref{lem:estimate-cylinder}, we have the estimate
\begin{align}
\begin{aligned}
&\int_{\R^d \times [0,\infty)} \phi(x,v) \big( \chi_{\hyp u(\tilde{t})} - \chi_{\mathsf{FT}(\hyp u(\tilde{t}-\tilde{a});\tilde{a})} \big) \,\dx x \dx v \\
&\quad =
\int_{\mathcal{Q}} \phi(x,v) \big( \chi_{\hyp u(\tilde{t})} - \chi_{\mathsf{FT}(\hyp u(\tilde{t}-\tilde{a});\tilde{a})} \big) \,\dx x \dx v 
+
\int_{\supp \phi \setminus \mathcal{Q}} \phi(x,v) \big( \chi_{\hyp u(\tilde{t})} - \chi_{\mathsf{FT}(\hyp u(\tilde{t}-\tilde{a});\tilde{a})} \big) \,\dx x \dx v \\
&\quad =
\mathcal{L}^{d+1}\big(\mathcal{Q} \cap \hyp u(\tilde{t}) \big) - 
		\mathcal{L}^{d+1}\big(\mathcal{Q} \cap \mathsf{FT}(\hyp u(\tilde{t}-\tilde{a}) ; \tilde{a})\big)
+
\int_{\supp \phi \setminus \mathcal{Q}} \phi(x,v) \big( \chi_{\hyp u(\tilde{t})} - \chi_{\mathsf{FT}(\hyp u(\tilde{t}-\tilde{a});\tilde{a})} \big) \,\dx x \dx v \\
& \quad \geq 
\frac{\omega h \abs{Q_r}}{2(d+1)} - \mathcal{L}^{d+1}\big( \supp \phi \setminus \mathcal{Q} \big) .
\end{aligned}
\end{align}
Clearly, we have
\begin{align}
	\mathcal{L}^{d+1}\big( \supp \phi \setminus \mathcal{Q} \big) \leq 2^d (r + r')^d (\omega + 2\omega') - 2^d r^d \omega
	\leq 2^d (2 r^d \omega' + 3 \cdot 2^d r^{d-1} r' \omega ) .
\end{align}
Therefore, so long as \(2^d (2 r^d \omega' + 3 \cdot 2^d r^{d-1} r' \omega ) \leq \frac{1}{4} \frac{\omega h \abs{Q_r}}{2(d+1)}\), we have
\begin{align}
	\int_{\R^d \times [0,\infty)} \phi(x,v) \big( \chi_{\hyp u(\tilde{t})} - \chi_{\mathsf{FT}(\hyp u(\tilde{t}-\tilde{a});\tilde{a})} \big) \,\dx x \dx v
	\geq \frac{\omega h \abs{Q_r}}{4(d+1)} .
\end{align}
It is sufficient to take \(r' \defeq \frac{h r}{3 \cdot 2^{d+4} (d+1)} \leq r\), and 
\(\omega' \defeq \frac{\omega h}{32(d+1)} \leq \omega\).

\textbf{2.}
Combining the above estimate with Lemma~\ref{lem:test-function-estimate}, we have
\begin{align}\label{e:altrettanto}
	\frac{\omega h \abs{Q_r}}{4 (d+1)} 
	\leq 
	\abs{\tilde{a}} \norm{g}_{L^\infty} \abs{Q_{r + r' + \abs{\tilde{a}}(\frac{\omega}{2}+\omega')\norm{f''}_{L^\infty}}} .
\end{align}
Estimating \(\abs{\tilde{a}}\) via~\eqref{eq:estimate-of-tildes}, and using the choice of \(\omega = C_\omega h^{1+\frac{d}{\alpha}}\), we have
\begin{align}
\frac{C_\omega h^{2+\frac{d}{\alpha}} r^d}{4 (d+1)} 
	\leq 
\widetilde{C} \frac{2 r \sqrt{d}}{h^{d/\alpha}} \norm{g}_{L^\infty} (2r + 2 \widetilde{C} 2 r \sqrt{d} C_\omega h \norm{f''}_{L^\infty})^d .
\end{align}
Finally, after rearranging and noting that \(h \in [0,\frac12]\), we obtain
\begin{align}
h \leq 
C_0 r^{\gamma_0},
\quad \ \,
C_0 \defeq \big[ \frac{\widetilde{C}}{C_\omega} 2^{d+3} (d+1)  \sqrt{d} \, \norm{g}_{L^\infty} (1 + \widetilde{C} C_\omega \sqrt{d} \, \norm{f''}_{L^\infty})^d \big]^{\gamma_0},
\quad \ \,
\gamma_0 \defeq \frac{1}{2(1+\frac{d}{\alpha})} .
\end{align}

%

\qed


Next we show that an algebraic decay rate of \(r \mapsto h_r(t,x)\), uniformly in \(x \in \R^d\), is sufficient to deduce the H{\"o}lder regularity in space of \(u(t,\cdot)\).

\begin{lemma} \label{lem:md-h-to-holder}
Suppose \(u \in C(\R\times\R^d;[0,1])\) satisfies the estimate
\begin{align}\label{e:assumption}
	h_r(t,x) \leq C r^\gamma 
\end{align}
for some \(\gamma \in (0,1)\), and \(C > 0\), uniformly in \(x \in \R^d\).
Then \(u(t,\cdot)\) is \(\gamma\)-H{\"o}lder with constant at most \(2 C \big( 1 + \frac{2}{2^\gamma -1} \big)\), uniformly in \(x \in \R^d\).
{
In particular, let $\gamma_0= \frac1{2(1+\frac{d}\alpha)}$ be as in the previous lemma.
Then there is $C_{\gamma_0}>0$ depending only on $\gamma_0$ such that, if $u$ satisfies \eqref{e:assumption} for some $\gamma \in [\gamma_0,1)$,  then $u(t,\cdot)$ is $\gamma$-H{\"o}lder with constant at most $C_{\gamma_0}C$.}
\end{lemma}

\textit{Proof of Lemma~\ref{lem:md-h-to-holder}.}
Since we work at a fixed time \(t \in \R\) throughout, let us suppress all time dependence in this proof.
For any \(x, y \in \R^d\), let \(r \defeq \abs{x-y}\).
We have
\begin{align}
\begin{aligned}
	\abs{u(x) - u(y)} 
	\leq 
	\abs{u(x) - (u)_{Q_r(x)}} + \abs{(u)_{Q_r(x)} - (u)_{Q_r(y)}} + \abs{u(y) - (u)_{Q_r(y)}} \\
	\leq
	\abs{u(x) - (u)_{Q_r(x)}} + \abs{u(y) - (u)_{Q_r(y)}} + 2 h_r(x) .
\end{aligned}
\end{align}
Now, since a single cube is tessellated by \(2^d\) sub-cubes each of half the side length, it is clear that 
\begin{align}
	(u)_{Q_r(x)} = \frac{1}{2^d} \sum_{i = 1}^{2^d}  (u)_{Q_{\frac{r}{2}}(x_i)} ,
\end{align}
where the points \(\{x_i \in \cl{Q_{\frac{r}{2}}(x)} : 1 \leq i \leq 2^d \}\) are the corners of the cube \(Q_{\frac{r}{2}}(x)\).
Therefore we can estimate
\begin{align}
\begin{aligned}
	\abs{(u)_{Q_r(x)} - (u)_{Q_{\frac{r}{2}(x)}}} 
	= \tabs{ \frac{1}{2^d} \sum_{i = 1}^{2^d} (u)_{Q_{\frac{r}{2}}(x_i)} - (u)_{Q_{\frac{r}{2}(x)}}}
	\leq \frac{1}{2^d} \sum_{i=1}^{2^d} \tabs{ (u)_{Q_{\frac{r}{2}}(x_i)} - (u)_{Q_{\frac{r}{2}}(x)} } \\
	\leq 2 h_{\frac{r}{2}}(x) .
\end{aligned}
\end{align}
Since \(u\) is continuous, we have
\begin{align}
\begin{aligned}
    \abs{(u)_{Q_r(x)} - u(x)} = \tabs{\sum_{k=0}^\infty \Big( (u)_{Q_{\frac{r}{2^k}}(x)} - (u)_{Q_{\frac{r}{2^{k+1}}}(x)} \Big) }
    \leq \sum_{k = 0}^\infty 2 h_{\frac{r}{2^{k+1}}}(x) 
    \leq 2 C \sum_{k=0}^\infty \big( \frac{r}{2^{k+1}} \big)^\gamma
    \\
    = \frac{2 C r^\gamma}{2^\gamma - 1 } .
\end{aligned}
\end{align}
Therefore 
\begin{align}
    \abs{u(x) - u(y)}  \leq 2 h_r (x) + \frac{4 C r^\gamma}{2^\gamma -1}
    \leq 2 C \big( 1 + \frac{2}{2^\gamma -1} \big) \abs{x-y}^\gamma .
\end{align}
\qed

Finally, combining Lemmas~\ref{lem:md-h-estimate}--\ref{lem:md-h-to-holder}, we can bootstrap our estimate of \(h_r\) using the additional H{\"o}lder regularity to improve the estimate.

{
\begin{lemma} \label{lem:md-bootstrap}
	Let \(u \in C(\R\times\R^d ; [0,1])\) be as in Lemma~\ref{lem:md-h-estimate}, and suppose there exist constants \(C_n >0\), and \(\gamma_n \in [\gamma_0,1)\) such that
	\begin{align}\label{e:ass-h}
		h_r(t,x) \leq C_n r^{\gamma_n}
	\end{align}  
	uniformly in \((t,x) \in \R\times\R^d\).
	Then there is a constant $C>0$ depending on $d,\norm{f''}_{L^\infty}, \norm{g}_{L^\infty},\widetilde{C}, C_\omega$ such that 
	\begin{align}\label{e:espIter}
		h_r(t,x) \leq C_{n+1} r^{ \gamma_{n+1}},
		\qquad \gamma_{n+1} \defeq \frac{1+2\gamma_n}{2(1 + \frac{d}{\alpha})}, \qquad C_{n+1}\le CC_n^{\frac1{1+\frac{d}\alpha}},
	\end{align}
    for some constant \(C > 0\), uniformly in \((t,x) \in \R\times\R^d\).
\end{lemma}

\textit{Proof of Lemma~\ref{lem:md-bootstrap}.}
The strategy of the proof is the same as the one of Lemma~\ref{lem:improved-holder}.
In particular, we deduce from \eqref{e:ass-h} and Lemma \ref{lem:md-h-to-holder}, that $u(t,\cdot)$ is $\gamma_n$-H{\"o}lder with constant at most $C_{\gamma_0}C_n$. 
We use this information to improve the estimate \eqref{eq:estimate-cylinder}.
In particular, estimate \eqref{e:1} still holds when we replace the integral in $v$ on the interval $[0, 1-h]$ with the one on the interval $[u(t,y_1)-C_{\gamma_0}C_n(\mathrm{diam}\ Q_r)^{\gamma_n}, u(t,y_1)+C_{\gamma_0}C_n(\mathrm{diam}\ Q_r)^{\gamma_n}]$. Then, the application of the mean value theorem gives the existence of a value $\overline v$ {that satisfies, rather than~\eqref{eq:mean-value-estimate2}, the finer estimate} 
\begin{equation}\label{e:improved-levels}
m(t,y_1,\overline v + h)-m(t,y_2,\overline v) \ge \frac{h|Q_r|}{C_n C_{\gamma_0}C_d r^{\gamma_n}}.
\end{equation}
\\
{
Now we replace \eqref{eq:mean-value-estimate2} with the improved estimate above while following of the proof of Lemma~\ref{lem:estimate-cylinder}:
we conclude that \eqref{eq:estimate-cylinder} can be improved to 
\[
\mathcal{L}^{d+1}\big(\mathcal{Q} \cap \hyp u(\tilde{t}) \big) - 
		\mathcal{L}^{d+1}\big(\mathcal{Q} \cap \mathsf{FT}(\hyp u(\tilde{t}-\tilde{a}) ; \tilde{a})\big)
		\geq \frac{\omega h_r(t,x) \abs{Q_r}}{C_nC_dC_{\gamma_0}r^{\gamma_n}},
\]
for an appropriate choice of $\omega$.
Observe that \eqref{e:improved-levels} improves \eqref{eq:m-prime-estimate}, and thanks to \eqref{e:altro}, it allows to choose the larger value
\begin{equation}\label{e:choice-gamma}
\omega = \widetilde{C}_\omega \frac{h_r(t,x)^{1+\frac{d}\alpha}}{C_n r^{\gamma_n}}
\end{equation}
for some $\widetilde{C}_\omega$ depending on $d, \widetilde{C}, \norm{f''}_{L^\infty}, C_{\gamma_0}$ but not on $n$.
 Concatenating this estimate with \eqref{eq:test-function-estimate} as in the proof of Lemma \ref{lem:md-h-estimate} we obtain in place of~\eqref{e:altrettanto}
 \begin{align}
\frac{\omega h_r(t,x) \abs{Q_r}}{C_nC_dC_{\gamma_0}r^{\gamma_n}}
 	\leq 
	\abs{\tilde{a}} \norm{g}_{L^\infty} \abs{Q_{r + r' + \abs{\tilde{a}}(\frac{\omega}{2}+\omega')\norm{f''}_{L^\infty}}} .
\end{align}
 Estimating \(\abs{\tilde{a}}\leq  \widetilde{C} \frac{2 r \sqrt{d}}{h_r(t,x)^{d/\alpha}} \) via~\eqref{eq:estimate-of-tildes}, and choosing $\omega$ as in \eqref{e:choice-gamma}, we have
 \[
 h_r(t,x)^{2(1+\frac{d}\alpha)}\le C C_n^2r^{1+2\gamma_n}
 \]
for some constant $C>0$ depending on $d,\norm{f''}_{L^\infty}, \norm{g}_{L^\infty},\widetilde{C}, C_\omega$,
and this proves the claim with $\gamma_{n+1}=\frac{1+2\gamma_n}{2(1+\frac{d}\alpha)}.$ }
\qed


\vskip\baselineskip

\vskip\baselineskip
The following statement is the main result of this section.



\begin{proposition} \label{cor:md-holder}
	Let \(u \in C(\R\times\R^d ; [0,1])\) be as in Lemma~\ref{lem:md-h-estimate}.
	Then for any \(t\in \R\), \(u(t,\cdot)\) is \(\overline{\gamma}\)-H{\"o}lder continuous in space for \(\overline{\gamma} \defeq \frac{\alpha}{2d}\). 
\end{proposition}

\textit{Proof of Proposition~\ref{cor:md-holder}.}
The constant \(\overline{\gamma}\) is given as the fixed point of the map \(\gamma_n \mapsto \gamma_{n+1}\), that is
\begin{align}\label{e:fixedpoint}
	\overline{\gamma} = \frac{1 +  2\overline{\gamma}}{2(1 + \frac{d}{\alpha})} ,
\end{align}
resulting in the stated value \(\overline{\gamma} \defeq \frac{\alpha}{2d}\). 
It remains to check that the H{\"o}lder constants with H{\"older} exponent $\gamma_n$ remain uniformly bounded as $n\to \infty$: since the map $y\mapsto Cy^{\frac1{1+\frac{d}\alpha}}$ is sublinear, then it follows from \eqref{e:espIter} that the sequence $C_n$ is bounded: it follows from Lemma \ref{lem:md-h-to-holder} that the H{\"o}lder constants are uniformly bounded as well.
\qed}



\begin{example}[The Burgers equation in \((1+d)\)-dimensions]
	For the Burgers equation in \((1+d)\)-dimensions, we have \(\alpha = \frac1{d}\), and therefore Proposition~\ref{cor:md-holder} gives us \(\overline{\gamma} = \frac{1}{2d^2}\). 
    This result is not optimal: the analysis in one space dimension, see e.g.~the analysis of \S\ref{S:Burg1d}, suggests that the optimal exponent is \(\overline{\gamma} = \frac{1}{d+1}\). 
    In the next section, we improve the exponent to $\overline{\gamma}=\frac 1 {2d}$.
\end{example}


%

\section{Improved decomposition lemma}

\label{sec:mdcl}

In this section we improve the decomposition Lemma~\ref{lem:decomposition-lemma} to obtain a better order on the estimate of the coefficients in \eqref{eq:estimate-of-coefficients}.
We prove it first in the case of the multi-dimensional Burgers equation, and then for general smooth nondegenerate fluxes $f$, in which cases we improve the order from $\frac{1}{d^2}$ to $\frac 1 d$.
This is interesting in order to get a better regularity estimate for continuous solutions to a balance law with bounded source, namely improving the H\"older continuity exponent from $\frac{1}{2 d^2}$ 
to $\frac 1 {2d}$ as stated in Proposition~\ref{cor:md-holder2} below.

In the direction approximately parallel to the $(1+\ell)$-th derivative $f^{(1+\ell)}$, the estimate further improves: up to the order $\frac 1 \ell$ for the decomposition lemma.

%

We first recall the elementary result in \cite[\S~2.3]{Silvestre} which relates Assumption \ref{ass:nl1} with Assumption \ref{ass:nl2}.

\begin{lemma}\label{L:Silv}
When $f\in C^{d+2}(\R)$, Assumption~\ref{ass:nl2} is equivalent to taking \(\alpha = \frac{1}{d}\) in Assumption~\ref{ass:nl1}.
\end{lemma}


In the following sub-sections we prove the following lemma, that, as stated, improves the decomposition Lemma~\ref{sec:mdcl} from the bound $h^{-1}$ raised to the power $\frac{1}{ d^2}$, to the better power ${\frac 1 d}$, in the worst case, and to the power ${\frac 1 \ell}$ in the direction $a_{\ell}$, for linearly independent directions $a_{1}, \dots, a_{\ell}$.

\begin{lemma} \label{lem:decomposition-lemma2}
Suppose the flux \(f \in C^{d+2}([0,1];\R^d)\) satisfies the nonlinearity Assumption~\ref{ass:nl2}.
	Given \(v > 0\) and \(0< h < 1\), set $v_{i}=v+\frac{ih}{d}$, where $ i=0, \,\dots, d$. Then for every \(a \in \R^d\), there exist \(\lambda=(\lambda_1, \,\cdots, \lambda_{d}) \in \R^d\) with
	\begin{align} \label{eq:estimate-of-coefficientsBur}
		a = \sum_{i=1}^{d} \lambda_i\left[ f'(v_i)-f'(v_0)\right],
		\qquad
		\norm{\lambda} \leq \widetilde{C}_{d} \frac{\norm{a}}{h^{d}} .
	\end{align}	
	Moreover, for each $\ell=1,\dots,d$, there is a direction $a_{\ell} \in \R^d$ satisfying the better estimate \[\norm{\lambda} \leq \widetilde{C}_{d} \frac{\norm{a_{\ell}}}{h^{\ell}} .\]
\end{lemma}

\begin{remark}
We show in the proof that for the Burgers equation, the direction $a_{\ell}$ is given by 
\[
a_{\ell}~=~f^{(1+\ell)}(v)~=~\left(\frac{i!}{(i-\ell)!}\cdot v^{i-\ell} \delta_{i\geq \ell}\right)_{i=1,\dots,d} .
\]
Then for a flux satisfying Assumption~\ref{ass:nl2}, $a_{\ell}$ is still a direction close to the direction of $f^{(1+\ell)}(v)$.
\end{remark}

To better illustrate the statement of the lemma, we provide as examples the explicit coefficients of the decomposition~\eqref{eq:estimate-of-coefficientsBur} in space dimensions $1, 2$ and $3$.

\begin{example}[$1d$ Burgers equation]
Consider a continuous function $u:\Omega\subseteq \R^{2}\to\R$ that satisfies
\[
  \pd_t u + \ddiv_{\bm{x}}\Big(\frac{u^{2}}{2}\Big) = g,
  \qquad
  d=1,\ 
  \alpha=\frac 1 {d}=1 ,
  \qquad f(u)=\frac{u^{2}}2, \ f'(u)=u, \ 
  f''(u)=1.
\]
Then given $a\in\R^{d}=\R$ and an interval $[v,v+h]$ we write the decomposition~\eqref{eq:estimate-of-coefficientsBur} explicitly as
\begin{align*}
&a= \frac{a}{h}\left( f'(v+h)-  f'(v)\right)= \frac{a}{h} (v+h)-  \frac{a}{h} v= \frac{a}{h}\cdot h,
\\
& \lambda_{1}= \frac{a}{h},
\qquad v_{0}=v,
\qquad v_{1}=v+h.
\end{align*}
Then clearly \(\abs{\lambda_1} \leq  \frac{\abs{a}}{h}\), as obtained in the previous Lemma~\ref{lem:decomposition-lemma} with \(d = 1\) and \(\alpha = 1\).
This leads to the optimal H{\"o}lder regularity of order \(\frac1 {2}\).

%
\end{example}

\begin{example}[$2d$ Burgers equation]
Consider a continuous function $u:\Omega\subseteq \R^{3}\to\R$ that satisfies
\[
  \pd_t u + \ddiv_{\bm{x}}\left(\frac{u^{2}}2,\frac{u^{3}}3\right)  = g ,
  \qquad
  d=2,  \ \alpha=\frac12, \]
  \[
f(u)=\Big(\frac{u^{2}}2,\frac{u^{3}}3\Big) ,  \quad  
f'(u)=( u,{u^{2}}), \quad 
f''(u)=( 1,{2u}) ,  \quad  
f'''(u)=( 0,2) .
\]
Then given $a=(a_x,a_y)\in\R^{d}=\R^{2}$, and an interval $[v,v+h]$, we set \[v_0 \defeq v,
\qquad v_1 \defeq v + \frac{h}{2},
\qquad v_2 \defeq v + h\] and we have $a= \sum_{i=1}^{2} \lambda_i\left( f'(v_{i+1})- f'(v_1)\right)$ for the choice 
\begin{align} 
	\lambda_1 \defeq  \frac{4  (h+2 v)a_x-4a_y}{h^2},
	\qquad
	\lambda_2 \defeq \frac{2 a_y- (h+4 v)a_x}{h^2}.
\end{align}
Then clearly \(\abs{\lambda_i} \leq C \frac{\abs{a}}{h^2}\), whereas the previous Lemma~\ref{lem:decomposition-lemma} with \(d = 2\) and \(\alpha = \frac12\) gives us the weaker estimate \(\abs{a_i} \leq C \frac{\abs{a}}{h^4}\).
Following through the whole argument with this improved estimate results in a better H{\"o}lder regularity of order \(\frac1 {4} \) in place of $\frac 1 {8}$.
\end{example}

\begin{example}[$3d$ Burgers equation]
Consider a continuous function $u:\Omega\subseteq \R^{3}\to\R$ that satisfies
\[
  \pd_t u + \ddiv_{\bm{x}}\left(\frac{u^{2}}2,\frac{u^{3}}3,\frac{u^{4}}4\right)  = g ,
  \qquad
  d=3,  \ \alpha=\frac13, \]
\begin{align*}
&f(u)=\Big(\frac{u^{2}}2,\frac{u^{3}}3,\frac{u^{4}}4\Big) ,   
&&f'(u)=( u,{u^{2}},u^{3}),   
&&f''(u)=( 1,{2u^{}}, 3u^{2}),  \\
& f'''(u)=( 0,{2 {}},6u),  
&& f'''(u)=( 0,{0 {}},6).&&
\end{align*}
Then given $a=(a_x,a_y,a_{z})\in\R^{d}=\R^{3}$, and an interval $[v,v+h]$, we set \[v_0 \defeq v,
\qquad v_1 \defeq v + \frac{h}{3},
\qquad v_2 \defeq v + \frac{2h}{3},
\qquad v_3 \defeq v + h,\] and we have $a= \sum_{i=1}^{3} \lambda_i\left( f'(v_{i+1})- f'(v_1)\right)$ for the choice 
\begin{align} 
	&\lambda_1 \defeq  \frac{9 \left((2 h^2+10   h v+9   v^2a_x  ) - (5   h+9    v) a_y +3 a_z \right)}{2 h^3},
	\\
	&\lambda_2 \defeq -\frac{9 \left(\left(h^2+8 h v+9 v^2\right)a_x  -  (4 h+9 v)a_y+3 a_z \right)}{2 h^3},
	\\
	&\lambda_3 \defeq  \frac{ \left(2 h^2+18 h v+27 v^2\right)a_x-9   (h+3 v)a_y+9 a_z }{2 h^3}.
\end{align}
Then clearly \(\abs{\lambda_i} \leq C \frac{\abs{a}}{h^3}\), whereas the previous Lemma~\ref{lem:decomposition-lemma}  with \(d = 3\) and \(\alpha = \frac13\) gives us the weaker estimate \(\abs{a_i} \leq C \frac{\abs{a}}{h^{9}}\).
Following through the whole argument with this improved estimate results in a better H{\"o}lder regularity of order \(\frac1 {6}\) in place of $\frac 1 {18}$.
\end{example}


Before we proceed with the proof of Lemma~\ref{lem:decomposition-lemma2}, we highlight that it is an interesting improvement because it yields an improved H\"older continuity of $u$.
That is, the exponent $\overline{\gamma} =\frac{1}{2 d^2}$ that was obtained in the preceding section in the setting of nonlinear fluxes satisfying Assumption~\ref{ass:nl1}, is now improved to the sharper estimate $\overline{\gamma} ={\frac{1}{2d}}$, when considering nonlinear fluxes satisfying Assumption~\ref{ass:nl2}.

\begin{proposition} \label{cor:md-holder2}
Suppose the flux \(f \in C^{d+2}([0,1];\R^d)\) satisfies the nonlinearity Assumption~\ref{ass:nl2}.
	Let \(u \in C(\R\times\R^d ; [0,1])\) solve~\eqref{eq:cl} in the sense of distributions, with \(g \in L^\infty(\R\times\R^d)\).
	Then, for any \(t \), \(u(t,\cdot)\) is \(\overline{\gamma}\)-H{\"o}lder continuous in space for \(\overline{\gamma} \defeq{ \frac{1}{2d}}\).
\end{proposition}

\vskip\baselineskip

\begin{proof}
When~\eqref{eq:estimate-of-coefficients} holds with $h^{d}$ replacing $h^{d/\alpha}$, then, in turn, with identical proofs:
\begin{itemize}
\item Equation~\eqref{eq:estimate-of-tildes} holds with $h^{d}$ instead of $h^{\frac{d}\alpha}$ and Equation~\eqref{eq:estimate-cylinder} holds for any \(\omega \in (0, C_\omega \, h_r(t,x)^{1+d}]\), replacing the previous statement of Equation~\eqref{eq:estimate-cylinder} that holds for any \(\omega \in (0, C_\omega \, h_r(t,x)^{1+\frac{d}{\alpha}}]\).

 \item Equation~\eqref{e:accessorio} holds with 
 $\gamma_0\defeq \frac1{2(1+d)}$ replacing 
 $	\gamma_0 \defeq \frac{1}{2(1+ \frac{d}{\alpha})}
=  \frac{1}{2(1+ {d^2})}
$.
 \item Equation~\eqref{e:espIter} holds with ${\overline\gamma_{n+1}} = \frac{1 +  2{\overline\gamma_n}}{2(1 + d)}$ 
 replacing  
 $\gamma_{n+1} \defeq \frac{1+2\gamma_n}{2(1 + \frac{d}{\alpha})}=\frac{1+2\gamma_n}{2(1 + {d^2}{})}$.

\item Finally, Corollary~\ref{cor:md-holder} holds with \(\overline{\gamma} \defeq \frac{1}{2 d}\) 
replacing \(\overline{\gamma} \defeq \frac{\alpha}{2d}=\frac{1}{2 d^2}\) since this is the fixed point of \[
	\overline{\gamma} = \frac{1 +  2\overline{\gamma}}{2(1+ d)} , 
\]
in place of looking for the fixed point of~\eqref{e:fixedpoint}.
\end{itemize}
\end{proof}

\begin{remark}When instead $f\in C^{k_{d}+1}(\R)$ and \(\{f^{(k_{1})}(v), \cdots, f^{(k_{d})}(v)\}\) is a spanning set of \(\R^d\), for some $d$ integers $2\leq k_{1}<\dots<k_{d}$, then a similar procedure is expected to lead to \(\overline{\gamma}\)-H{\"o}lder continity in space for \(\overline{\gamma} \defeq \frac{1}{2(k_d-1)}\).
\end{remark}


\subsection{Proof for Multi-Dimensional Burgers}
In this section we prove Lemma~\ref{lem:decomposition-lemma2} for a continuous function $u:\Omega\subseteq \R^{d+1}\to\R$ that satisfies the particular, but particularly interesting, case of the Burgers equation:
\begin{subequations}\label{E:BurgersMD}
\begin{align}
  &\pd_t u + \ddiv_{\bm{x}}\left(\frac{u^{2}}2\,,\frac{u^{3}}3\,,\dots\,,\frac{u^{(d+1)}}{d+1}\right) ~=~ g \,,
\\
 &f(u) ~=~ \bigg(\frac{u^{2}}2\,,\frac{u^{3}}3\,,\dots\,,\frac{u^{(d+1)}}{d+1}\bigg) \,, \qquad f'(u)~=~( u\,,{u^{2}}\,,\dots\,, u^{d}) \,,
\\
& f^{(1+\ell)}(u)~=~\left(\frac{i!}{(i-\ell)!}\cdot u^{i-\ell} \delta_{i\geq \ell}\right)_{i=1,\dots,d} \,.
\end{align}
\end{subequations}
 
We begin with an estimate of the operator norm of a particular matrix. 
First, we recall the following definition of norms, see \cite[Definitions~8.6-7]{matrici}.
 
 \begin{recall}\label{D:matrixNorm}
 Given a matrix $A\in \M_{d\cross d}$, where $A=\left[a_{i\ell}\right]_{i,\ell=1,\dots,d}$, its operator norm is defined as
\begin{equation}\label{E:opNorm}
\norm{ A }~\defeq~ \max_{\{\norm{v}=1\}} \norm{A v}{}~=~ \max_{\{v\neq 0\}} \frac{\norm{A v}}{\norm{v}} \,,
\end{equation}
 where we consider the Euclidean norm on $\R^{d}$. The Frobenius norm is defined as
 \begin{equation}\label{e:FrobNorm}
 \norm{A}_{F}
 ~\defeq~
\sqrt{\sum_{i,\ell=1}^{d}a_{i\ell}^{2}} \,.
\end{equation}
 We also recall that the inverse of a matrix $A$ with non vanishing determinant is expressed as
 \begin{equation}\label{E:inverseMatrA}
A^{-1}
~=~ 
\frac{1 }{\deter A}\adj(A )
\end{equation}
where
\begin{itemize}
\item $\deter A$ is the determinant of the matrix $A$ and
\item the element $a_{i\ell}$ at position $(i,\ell)$ of $\adj(A)$ the determinant of the sub-matrix obtained after suppressing from $A$ the $\ell$-th row and $i$-th column, multiplied by $(-1)^{\ell+i}$.
\end{itemize}
\end{recall}
  
  \begin{remark}
  The operator norm and Frobenius norm are equivalent, see~\cite[Proposition 8.4]{matrici}:
\begin{equation}\label{e:nomrEq} \norm{A}~\leq~ \norm{A}_{F}~\leq~\sqrt{d} \norm{A} \qquad \text{for all $A\in M_{n\times n}\,$.}\end{equation}
\end{remark}
 
We now provide an auxiliary estimate for a relevant matrix.

 \begin{lemma}\label{L:normaH}
Let $d\in\N$ and $0<h\leq 1$. The operator norm of the inverse of the matrix
\[ 
H_d~=~H_d(h)~\defeq~\left[
\left(\frac{ih}{d}\right)^{\ell}
\right]_{\ell,i=1,\dots,d}
\]
can be estimated as $\norm{H_d^{-1}}(h)\leq C_d h^{-d}$, for a positive dimensional constant $C_d$.
 \end{lemma}

 \begin{proof}
 Recall first from~\eqref{E:inverseMatrA} the general expression of the inverse matrix
 \begin{equation}\label{E:inverseMatr}
H_d^{-1}
~=~ 
\frac{1 }{\deter H_d}\adj(H_d )\, .
\end{equation}

By multi-linearity of the determinant, one can extract from the $\ell$-th row the common factor $\left({\frac h d}\right)^\ell$ and repeat this for all rows $\ell=1,\dots,d$: multiplying the common factors we extracted, we get\footnote{The determinant of $\left[
{i}^{\ell}
\right]_{\ell,i=1,\dots,d}$ can be computed by induction, it is related to one of the Vandermonde matrices, see~\cite[Example 6.2]{matrici}.}
\begin{equation}\label{E:inverseMatr2}
\frac{1 }{\deter H_d}~=~
\frac{1}{\deter{\left[
{i}^{\ell}
\right]_{\ell,i=1,\dots,d}}}
\cdot\left(\frac d h\right)^{\frac{(d+1)d}{2}} ,
\qquad
\deter{\left[
{i}^{\ell}
\right]_{\ell,i=1,\dots,d}}~=~
  \prod_{0\leq i \leq d}  i ! \,.
\end{equation}

Again by multi-linearity of the determinant, and by definition of the adjoint matrix $\adj(H_d )$, in the elements of the adjoint matrix obtained suppressing the $\ell$-th row of $ H_d $, there is a common factor of $\frac h d$ with exponent $\frac{(d+1)d}{2}-\ell\geq \frac{(d+1)d}{2}-d$.
Therefore denoting $\adj(H_d )=\left[a_{ i\ell}\right]_{i,\ell=1,\dots,d}$ we have
\begin{align}\label{e:cofattore}
a_{ i\ell}~&=~(-1)^{i+\ell}
 {{\deter{\left[
{j}^{m}
\right]_{\substack{ m,j=1,\dots,d\\ m\neq \ell\,, j\neq i }}}}}\cdot 
{ \left(\frac h d \right)^{\frac{(d+1)d}{2}-\ell} }
\\
~&\leq~\notag
(-1)^{i+\ell} {{\deter{\left[
{j}^{m}
\right]_{\substack{ m,j=1,\dots,d\\ m\neq \ell\,, j\neq i }}}}}\cdot 
{ \left(\frac h d \right)^{\frac{(d+1)d}{2}-d} }\,.
\end{align}
By the definition of the Frobenius norm \eqref{e:FrobNorm}, thus we estimate, for a dimensional constant $ \widetilde C_d>0$,
\begin{equation}\label{E:acc2}
\norm{\adj(H_d )}_{F}
~ =~
\sqrt{\sum_{i,\ell=1}^{d}a_{i\ell}^{2}}
 ~\leq~
 \widetilde C_d\cdot 
  h^{\frac{(d+1)d}{2}-d} \,.
\end{equation}
Inserting estimates~\eqref{E:acc2} and~\eqref{E:inverseMatr2} into~\eqref{E:inverseMatr} we get 
\begin{equation}\label{E:inverseMatrHest}
\norm{H_d^{-1}}_{F}
~\leq~ 
C_p\cdot 
  h^{ -d}\,,
  \qquad 
  C_p~\defeq~\frac{ d^{\frac{(d+1)d}{2}}\widetilde C_d}{ \prod_{0\leq i \leq d}  i !} \,.
  \end{equation}
From the equivalence of the Frobenius norm and the operator norm~\eqref{e:nomrEq} we finally get the thesis.
 \end{proof}
 
 \begin{example}
 We write explicitly the matrices in low dimension $d=2,3,4$:
 \[
H_{2}(h)~=~ \left(
\begin{array}{cc}
 \frac{h}{2} & h \\
 \frac{h^2}{4} & h^2 \\
\end{array}
\right)
\qquad
H_{2}^{-1}(h)~=~
\left(
\begin{array}{cc}
 \frac{4}{h} & -\frac{4}{h^2} \\
 -\frac{1}{h} & \frac{2}{h^2} \\
\end{array}
\right)
\]
\[
H_{3}(h)~=~ \left(
\begin{array}{ccc}
 \frac{h}{3} & \frac{2 h}{3} & h \\
 \frac{h^2}{9} & \frac{4 h^2}{9} & h^2 \\
 \frac{h^3}{27} & \frac{8 h^3}{27} & h^3 \\
\end{array}
\right)
\qquad
H_{3}^{-1}(h)~=~
\left(
\begin{array}{ccc}
 \frac{9}{h} & -\frac{45}{2 h^2} & \frac{27}{2 h^3} \\
 -\frac{9}{2 h} & \frac{18}{h^2} & -\frac{27}{2 h^3} \\
 \frac{1}{h} & -\frac{9}{2 h^2} & \frac{9}{2 h^3} \\
\end{array}
\right)
\]\[
H_{4}(h)~=~\left(
\begin{array}{cccc}
 \frac{h}{4} & \frac{h}{2} & \frac{3 h}{4} & h \\
 \frac{h^2}{16} & \frac{h^2}{4} & \frac{9 h^2}{16} & h^2 \\
 \frac{h^3}{64} & \frac{h^3}{8} & \frac{27 h^3}{64} & h^3 \\
 \frac{h^4}{256} & \frac{h^4}{16} & \frac{81 h^4}{256} & h^4 \\
\end{array}
\right)
\qquad
H_{4}^{-1}(h)~=~
\left(
\begin{array}{cccc}
 \frac{16}{h} & -\frac{208}{3 h^2} & \frac{96}{h^3} & -\frac{128}{3 h^4} \\
 -\frac{12}{h} & \frac{76}{h^2} & -\frac{128}{h^3} & \frac{64}{h^4} \\
 \frac{16}{3 h} & -\frac{112}{3 h^2} & \frac{224}{3 h^3} & -\frac{128}{3 h^4} \\
 -\frac{1}{h} & \frac{22}{3 h^2} & -\frac{16}{h^3} & \frac{32}{3 h^4} \\
\end{array}
\right).
\]
 \end{example}
 
We can now finally prove Lemma~\ref{lem:decomposition-lemma2} in the case of the Burgers equation, where the computations are more explicit.
 
\textit{Proof of Lemma~\ref{lem:decomposition-lemma2}, Burgers.}
Given $a \in\R^{d} $, and an interval $[v,v+h]$, we want to write 
\[
a= \sum_{i=1}^{d} \lambda_i ( f'(v_i)- f'(v_0))\,, 
\qquad
\text{for} \quad 
v_{i}=v+\frac{ih}{d}\,, \quad i=0\,,\dots\,,d\,.
\]

Considering $f$, $a$, $\lambda$ as column vectors, then
\[
a~=~A\lambda
\qquad\Longleftrightarrow\qquad
\lambda~=~A^{-1}a
\]
for the square matrix $A\in\M_{d\times d}$ defined by the increments
\[
A(v,h,d)~\defeq~\left[ f'(v_1)- f'(v_0)| \, \dots \, | f'(v_d)- f'(v_0)\right]\,.
\]
For better understanding, we make explicit the matrix $A$ in small dimension.
If $d=4$ one has
\[
A(v,h,4)~=~\left(
\begin{array}{cccc}
\frac{h}{4} & \frac{h}{2} & \frac{3h}{4} & h \\
\left(\frac{h}{4}+v\right)^2-v^2 & \left(\frac{h}{2}+v\right)^2-v^2 & \left(\frac{3h}{4}+v\right)^2-v^2 & (h+v)^2-v^2 \\
\left(\frac{h}{4}+v\right)^3-v^3 & \left(\frac{h}{2}+v\right)^3-v^3 & \left(\frac{3h}{4}+v\right)^3-v^3 & (h+v)^3-v^3 \\
\left(\frac{h}{4}+v\right)^4-v^4 & \left(\frac{h}{2}+v\right)^4-v^4 & \left(\frac{3h}{4}+v\right)^4-v^4 & (h+v)^4-v^4 \\
\end{array}
\right)\]
which is, denoting by $A^{T}$ the transpose of $A$,
\[
A^{T}(v,h,4)~=~\left(
\begin{array}{cccc}
 \frac{h}{4} & \frac{h^2}{16}+\frac{h v}{2} & \frac{h^3}{64}+\frac{3 h^2 v}{16}+\frac{3 h v^2}{4} & \frac{h^4}{256}+\frac{h^3 v}{16}+\frac{3 h^2 v^2}{8}+h v^3 \\
 \frac{h}{2} & \frac{h^2}{4}+h v & \frac{h^3}{8}+\frac{3 h^2 v}{4}+\frac{3 h v^2}{2} & \frac{h^4}{16}+\frac{h^3 v}{2}+\frac{3 h^2 v^2}{2}+2 h v^3 \\
 \frac{3 h}{4} & \frac{9 h^2}{16}+\frac{3 h v}{2} & \frac{27 h^3}{64}+\frac{27 h^2 v}{16}+\frac{9 h v^2}{4} & \frac{81 h^4}{256}+\frac{27 h^3 v}{16}+\frac{27 h^2 v^2}{8}+3 h v^3 \\
 h & h^2+2 h v & h^3+3 h^2 v+3 h v^2 & h^4+4 h^3 v+6 h^2 v^2+4 h v^3 \\
\end{array}
\right).
\]
In general $A(v,h,d)=[ f'(v_1)- f'(v_0)| \,\dots\, | f'(v_d)- f'(v_0)]
$ is, indexing rows by $\ell$ and columns by $i$,
\[
A(v,h,d)~=~
\left[ \left(v-\frac{ih}{d}\right)^{\ell}- v^{\ell}\right]_{\ell,i=1,\dots,d}
~=~
\left[\sum_{k=0}^{\ell-1} \binom{\ell}{k}\left(\frac{ih}{d}\right)^{\ell-k} v^{k}\right]_{\ell,i=1,\dots,d}\,.
\]
Notice that
\[
\left(\begin{array}{cccc}
1 & 0 & 0 & 0 \\
-2v & 1 & 0 & 0 \\
3v^{2} & -3v & 1 & 0 \\
-4v^{3} & 6v^{2} & -4v & 1 \\
\end{array}\right)A(v,h,4)~=~
\left(
\begin{array}{cccc}
\frac{h}{4} & \frac{h}{2} & \frac{3h}{4} & h \\
\frac{h^2}{16} & \frac{h^2}{4} & \frac{9h^2}{16} & h^2 \\
\frac{h^3}{64} & \frac{h^3}{8} & \frac{27h^3}{64} & h^3 \\
\frac{h^4}{256} & \frac{h^4}{16} & \frac{81h^4}{256} & h^4 \\
\end{array}
\right).
\]
Similarly, indexing with $\ell$ rows and with $i$ columns, we have the expression
\[
\left[
\binom{\ell}{i}\left(-1\right)^{i+\ell} v^{\ell-i}\delta_{\ell\geq i}
\right]_{\ell,i=1,\dots,d}
A(v,h,d)
~=~
\left[
\left(\frac{ih}{d}\right)^{\ell}
\right]_{\ell,i=1,\dots,d}
\]
so that
\[
A(v,h,d)
~=~
\left[
\binom{\ell}{i} v^{\ell}\delta_{\ell\geq i}
\right]_{\ell,i=1,\dots,d}
\left[
\left(\frac{ih}{d}\right)^{\ell}
\right]_{\ell,i=1,\dots,d}
\]
and
\begin{align}\label{E:matrixProductInverseA}
\left(A(v,h,d)\right)^{-1}
~=~
\left(\left[
\left(\frac{ih}{d}\right)^{\ell}
\right]_{\ell,i=1,\dots,d}\right)^{-1}
\left[
\binom{\ell}{i}\left(-1\right)^{i+\ell}  v^{\ell-i}\delta_{\ell\geq i}
\right]_{\ell,i=1,\dots,d}.
\end{align}

Being $\lambda=A^{-1}a$, we want to understand the operator norm of $A^{-1}$, as defined in Definition~\ref{D:matrixNorm}. 
Since the operator norm is sub-multiplicative, then
\begin{equation}\label{E:stima1}
\norm{ A^{-1}(v,h,d)}~\leq~\norm{\left[
\binom{\ell}{i}\left(-1\right)^{i+\ell}  v^{\ell-i}\delta_{\ell\geq i}
\right]_{i,\ell=1,\dots,d}
}
\cdot  \norm{H_d^{-1}(h)}{} \,,
 \end{equation}
 where $H_d$ is precisely the matrix in Lemma~\ref{L:normaH}.
 Defining the polynomial
 \[
 p_{d}(v)~=~\norm{\left[
\binom{\ell}{i}\left(-1\right)^{i+\ell}  v^{\ell-i}\delta_{\ell\geq i}
\right]_{i,\ell=1,\dots,d}
}_F^2
 \]
 thus from~\eqref{E:stima1} and Lemma~\ref{L:normaH}, jointly with the equivalence of norms~\eqref{e:nomrEq}, we conclude
 \[
 \norm{ A^{-1}(v,h,d)}~\leq~ C_d \sqrt{p_d(v)}\cdot h^{-d}\,.
 \]
 
 Finally, be definition of the operator norm, one has the thesis~\eqref{eq:estimate-of-coefficientsBur}:
\[
\norm{\lambda}~=~\norm{\left(A(v,h,d)\right)^{-1} a} ~\leq~ C_{d}\sqrt{p_{d}(v)}
 \cdot\frac{\norm{a}}{h^d}\,.
 \]
 \qed

\textit{Proof of the stronger directional estimate in Lemma~\ref{lem:decomposition-lemma2} for the Burgers equation.}
In the setting of the proof just concluded for the general estimate in Lemma~\ref{lem:decomposition-lemma2}, define the matrix $W\in\M_{d\times d}$ whose columns are given by the direction of the derivatives $f''$,\dots,$f^{(d+1)}$ for the flux, up to coefficients, as
\begin{equation}\label{e:wroB}
W~=~\left[
\binom{\ell}{i} v^{\ell}\delta_{\ell\geq i}
\right]_{\ell,i=1,\dots,d}
~=~
\left(\begin{array}{cccccc}
1 & 0 & 0 & 0 & 0 \\
2v & 1 & 0 & 0 & 0 \\
3v^{2} & 3v & \ddots & 0 & 0 \\
4v^{3} & 6v^{2} & & 1 & 0 \\
5v^{4} & 10v^{3} & \cdots & dv & 1 \\
\end{array}\right).
\end{equation}
Consider now vectors $a\in\R^{n}$, $\norm{a}=1$, that satisfy the proportionality
\[
W^{-1}
a~=~\xi \, \re_{\ell}
 \qquad\Longleftrightarrow\qquad
 a~=~\xi \cdot W^{} \cdot\re_{\ell} \,,
\]
where necessarily $ \xi \geq C(v,d)$ since $W$ is not singular.
Then one has 
the expression
\[
\lambda  ~=~A^{-1}(v,h,d)a ~=~
    H_d^{-1}(h)\cdot W^{-1} a
~=~
\xi H_d^{-1}(h)\cdot  \re_{\ell}\,.
\]
Since, by its expression~\eqref{E:inverseMatr2}--\eqref{e:cofattore}, the $\ell$-th column of the matrix $ H_d^{-1}(h)$ is proportional to $h^{-\ell}$, then
\[
\norm{\lambda }~\geq~  \xi ~\geq~ C(v,d)\cdot \tilde c_{d}\cdot h^{-\ell}\,\,
\]
where $\tilde c_{d}$ is the minimum norm among columns of the matrix $ H_d^{-1}(1)$.
In particular, for vectors $a_{\ell}$ that are proportional to the vector
\[
 f^{(1+\ell)}(v)~=~\ell!~ W\cdot \re_{\ell} 
   \]
one has the stronger estimate in the thesis.
\qed

\subsection{Proof for Smooth Nondegenerate Fluxes}

We finally consider a smooth flux $f=(f_1,\dots, f_d)$ that satisfies the nonlinearity Assumption~\ref{ass:nl2}.

\vskip\baselineskip
 
 \textit{Proof of the general estimate in Lemma~\ref{lem:decomposition-lemma2}, nondegenerate fluxes.}
Given $a \in\R^{d} $, and an interval $[v,v+h]$, we want to write 
\[
a= \sum_{i=1}^{d} \lambda_i ( f'(v_i)- f'(v_0)) \,,
\qquad \text{for}
\quad 
v_{i}=v+\frac{ih}{d}
\quad i=0\,,\dots\,,d\,.
\]

Considering $f$, $a$, $\lambda$ as column vectors, then
\[
a~=~A\lambda
\qquad\Longleftrightarrow\qquad
\lambda~=~A^{-1}a
\]
for the square matrix $A(v_0,h,d)\in\M_{d\times d}$ defined by the increments
\begin{equation}\label{e:Aincrements}
A(v_0,h,d)~\defeq~[ f'(v_1)- f'(v_0)|\,\dots\, | f'(v_d)- f'(v_0)]\,.
\end{equation}
Having a smooth flux, we exploit the Taylor expansion
\begin{equation}
\label{E:Taylor}
 f'(v_i)- f'(v_0)~=~\sum_{\ell=1}^d\frac{1}{\ell!} f^{(1+\ell)}(v_0)\left(\frac{ih}{d}\right)^\ell + \sum_{j = 1}^{d}\frac{1}{(d+1)!}\re_j f_j^{(d+2)}(\xi_{j,i})  \left(\frac{ih}{d}\right)^{d+1}
\end{equation}
for some values  $\xi_{j,i}\in[v_0,v_i]$ and for $i,j=1,\dots, d$.

We can express the matrix of increments $A(v_0,h,d)$, essentially, as a product of a matrix $W(f,v_0)\in\M_{d\times d}$ depending only on the flux times a matrix $H_{d}(h)\in\M_{d\times d}$ depending only on the increments. To do so, consider the almost Wronskian matrix, where $\ell$ denotes rows and $i$ columns,
\begin{equation}\label{e:wro}
W(f,v_0)~=~ \left[ \frac{1}{\ell!} f_i^{(1+\ell)}(v_0)  \right]_{\ell,i=1,\dots,d}
\end{equation}
and define the remainder matrix $R(f,v_0,h)\in\M_{d \times d}$ as
\begin{equation}\label{e:reminder}
 R(f,v_0,h) ~\defeq~\sum_{j=1}^d\frac{1}{(d+1)!}\left[ \re_jf_j^{(2+d)}(\xi_{j,1})\Big|2^{(d+1)} \re_j f_j^{(2+d)}(\xi_{j,2})\Big|\dots\Big|d^{(d+1)} \re_jf_j^{(2+d)}(\xi_{j,d}) \right] .
\end{equation}
Then
\begin{align}
A(v_0,h,d)&~=~
\left[ \frac{1}{\ell!} f_m^{(1+\ell)}(v_0)  \right]_{m,\ell=1,\dots,d}
\left[
\left(\frac{ih}{d}\right)^{\ell-i}
\right]_{\ell,i=1,\dots,d}
+ R(f,v_0,h)  \left(\frac{h}{d}\right)^{d+1}\notag
\\
&~=~\label{E:ecco1}
  W(f,v_0) \cdot H_d(h)+ R(f,v_0,h)  \left(\frac{h}{d}\right)^{d+1} .
\end{align}

The non-degeneracy Assumption~\ref{ass:nl2} precisely guarantees that the matrix $W(f,v)$ is invertible: thus $\norm{  W(f,v_0) ^{-1} }\leq C(f,v_0)$ for some constant $C(f,v_0)$. Then, by sub-multiplicativity of the norm of matrices, we estimate the inverse of the approximation $  W(f,v_0) \cdot H_d(h)$ of $A(v_0,h,d)$ as
\begin{equation}\label{E:ecco3}
\norm{ \left[  W(f,v_0) \cdot H_d(h)\right]^{-1}}_F
~= ~
\norm{    H_d^{-1}(h)\cdot W^{-1}(f,v_0) }_F
~\leq ~
C(f,v_0) \cdot  C_d h^{-d}\,,
\end{equation}
thanks to the estimate of $\norm{H_d^{-1}(h)}$ in Lemma~\ref{L:normaH}.

It remains to prove that the reminder in~\eqref{E:ecco1} does not destroy such estimate.
To that extent, we apply that if $\widetilde A, B\in\M_{d\times d}$ are two matrices with $\norm{\widetilde A^{-1} R}<1$, see~\cite[Proposition 8.8]{matrici}, then
\begin{equation}\label{E:ecco2}
\norm{\left(\widetilde A+R\right)^{-1}}~\leq ~\frac{\norm{\widetilde A}}{1-\norm{\widetilde A^{-1} R}}
~=~\norm{\widetilde A} \cdot \left(1+\frac{\norm{\widetilde A^{-1} R}}{1-\norm{\widetilde A^{-1} R}}\right).
\end{equation}
Pick the two matrices $\widetilde A, R\in\M_{d\times d}$ defined as\[\widetilde A~\defeq~  W(f,v_0)\cdot H_d(h) \qquad\text{and}\qquad R~\defeq~R(f,v_0,h)  \left(\frac{h}{d}\right)^{d+1}\,.\] 
In order to apply this inequality~\eqref{E:ecco2}, we show that $\norm{\widetilde A^{-1}\cdot R}<\frac12$ for $h$ small: by~\eqref{E:ecco3} and by sub-multiplicity
\begin{equation}\label{E:stimaProdb}
\norm{\widetilde A^{-1}\cdot R}~\leq~
\norm{R(f,v_0,h) }\cdot C(f,v_0) \cdot \frac{C_d }{d^{d+1}}\cdot h\,.
\end{equation}
Since for $h=0$ the matrix $R(f,v_0,h) $ in~\eqref{e:reminder} has all columns proportional to $f^{(d+2)}(v_0)$, then
\[
\norm{R(f,v_0,0) }~=~\frac{1}{(d+1)!}\norm{f^{(d+2)}(v_0)}\cdot \norm{(1,2^{(d+1)},\dots,d^{(d+1)})}\,,
\] 
and $\norm{R(f,v_0,h) }$ is close to such value for $h$ small.
In particular, by~\eqref{E:stimaProdb} for $h$ small one has $\norm{\widetilde A^{-1}\cdot B}\leq 2^{-1}$, which allows to apply~\eqref{E:ecco2} to estimate the norm of the inverse of $A=\widetilde A+R$:
\begin{equation}\label{E:ecco2bis}
\norm{A^{-1}(v_0,h,d) }~=~\norm{\left(\widetilde A+R\right)^{-1}}~\leq ~ 2\norm{\widetilde A} .
\end{equation}
Recalling~\eqref{E:ecco1} we reach the thesis: for $h$ small it holds that
\[
\norm{A^{-1}(v_0,h,d) }~\leq~
\norm{\widetilde A^{-1}\cdot B}~\leq~
\widetilde 2C(f,v_0) \cdot  C_d h^{-d}\,.
\]
\qed

\textit{Proof of the stronger directional estimate in Lemma~\ref{lem:decomposition-lemma2} for nondegenerate fluxes.} 
Continuing with the notation of the proof just concluded for the general estimate in Lemma~\ref{lem:decomposition-lemma2}, in particular \eqref{e:reminder}, consider now vectors $a\in\R^{n}$, $\norm{a}=1$, that for some $\xi\in\R$ satisfy the proportionality
\[
\left(W(f,v_0)+H_{h}^{-1}(h) R(f,v_0,h)  \left(\frac{h}{d}\right)^{d+1} \right)^{-1}
a~=~\xi \, \re_{\ell}
\]
which is of course equivalent to
\[ a~=~\xi \cdot \left(W(f,v_0)+H_{h}^{-1}(h) R(f,v_0,h)  \left(\frac{h}{d}\right)^{d+1} \right)\cdot\re_{\ell} \,.
\]
Observe that necessarily $ \xi \geq C(v_0,d)$, for $h$ small enough, since $\left(W(f,v_0)+H_{h}^{-1}(h) R(f,v_0,h) \right)$ is not singular, since $W(f,v_0)$ is not singular by the nonlinearity Assumption~\ref{ass:nl2} and by Lemma~\ref{L:normaH}
\[
\norm{H_{h}^{-1}(h) R(f,v_0,h)  \left(\frac{h}{d}\right)^{d+1}}~\leq ~
\norm{H_{h}^{-1}(h) }\cdot \norm{R(f,v_0,h) }\cdot  \left(\frac{h}{d}\right)^{d+1}
~\leq~ C_d \norm{R(f,v_0,h)} h^{}\,.
\]
Then one has 
the expression
\[
  A^{-1}(v_0,h,d)a ~=~
    H_d^{-1}(h)\cdot \left(W(f,v_0)+H_{h}^{-1}(h) R(f,v_0,h)  \left(\frac{h}{d}\right)^{d+1} \right)^{-1} a
~=~
\xi H_d^{-1}(h)\cdot  \re_{\ell}\,.
\]
Since, by its expression~\eqref{E:inverseMatr2}--\eqref{e:cofattore}, the $\ell$-th column of the matrix $ H_d^{-1}(h)$ is proportional to $h^{-\ell}$, then
\[
\norm{  A^{-1}(v_0,h,d)a  }~\geq~  \xi ~\geq~ C(v_0,d)\cdot \tilde c_{d}\cdot h^{-\ell}\,\,
\]
where $\tilde c_{d}$ is the minimum norm among columns of the matrix $ H_d^{-1}(1)$.

In particular, for vectors $a_{\ell}$ that are proportional to the vector
\[
\frac{1}{\ell!}f^{(1+\ell)}(v_0)+\frac{2^{(d+1)}}{d^{d+1}\cdot(d+1)!}H_{h}^{-1}(h)  f^{(2+d)}(\xi_\ell)  \cdot  h^{d+1}\,,
   \]
which perturbs $\frac{1}{\ell!}f^{(1+\ell)}(v_0)$, one has the stronger estimate in the thesis.
\qed
\bigskip\bigskip

\noindent {\bf Acknowledgements.}  
F.~Ancona, L.~Caravenna and E.~Marconi  are members of GNAMPA of the ``Istituto Nazionale di Alta Matematica
F.~Severi'',  are partially supported by the PRIN
Project 20204NT8W4 “Nonlinear evolution PDEs, fluid dynamics and transport
equations: theoretical foundations and applications”,
and by the PRIN 2022 PNRR Project P2022XJ9SX ``Heterogeneity on the road---modeling, analysis, control''.
E.~Marconi  is also partially supported by H2020-MSCA-IF “A~Lagrangian approach: from conservation laws to line-energy Ginzburg-Landau models”.

\vspace{0.5cm}
\noindent {\bf Declaration of interests.}  
The authors report there are no competing interests to declare.

\vspace{1cm}

\vfill

\bibliographystyle{siam}


\end{document}